\documentclass{amsart}

\usepackage{latexsym}
\usepackage{amssymb, latexsym}
\usepackage{amsmath}
\usepackage{mathrsfs}
\usepackage{amsxtra}
\usepackage[balance]{diagrams}
\usepackage{amsthm}
\newtheorem{theorem}{Theorem}[section]
\newtheorem{corollary}[theorem]{Corollary}
\newtheorem{sublemma}{Lemma}[theorem]
\newtheorem{lemma}[theorem]{Lemma}

\newtheorem{proposition}[theorem]{Proposition}

\newtheorem*{theo}{Theorem}
\newtheorem*{pro}{Proposition}
\theoremstyle{definition}
\newtheorem{definition}[theorem]{Definition}
\newtheorem{question}[theorem]{Question}

\newcommand{\tmop}[1]{\ensuremath{\operatorname{#1}}}
\newcommand{\x}{\mathfrak{X}}

\newcommand{\la}{\langle}
\newcommand{\ra}{\rangle}
\newcommand{\sa}{\mathcal A}
\newcommand{\arr}[1]{\overset{\to}{#1}}
\newcommand{\ahat}[2]{\widetilde{\mathcal M}_{#1}^{#2}}
\newcommand{\ah}[2]{\mathcal M_{#1}^{#2}}
\newcommand{\her}[1]{H_{{#1}^+}}
\newcommand{\map}[3]{f_{#1#2}^{#3}}
\newcommand{\ult}[2]{U_{#1}^{#2}}
\newcommand{\ahh}[1]{\mathcal M_{#1}}
\newcommand{\ma}[2]{f_{#1#2}}

\newcommand{\lset}[1]{\langle L_{#1}[A],A\rangle}
\newcommand{\lsetr}[1]{\langle L_{#1}[A\cap #1],A\cap #1\rangle}
\newcommand{\lsetex}[1]{\langle L_{#1}[A],A,\xi\rangle}
\newcommand{\lone}[1]{L_{#1}[A]}

\newcommand{\Power}{\mathcal P}

\newcommand{\diagonal}{\bigtriangleup}
\newcommand{\restrict}{\upharpoonright}

\newcommand{\mhat}[2]{\widetilde{M}_{#1}^{#2}}
\newcommand{\mh}[2]{M_{#1}^{#2}}
\newcommand{\mhh}[1]{M_{#1}}

\newarrow{Corresponds}<--->
\newarrow{Dashto}{}{dash}{}{dash}>
\newarrow {Dotsto} .....

\begin{document}
\author{Victoria Gitman}
\today
\address{New York City College of Technology (CUNY), 300 Jay Street,
Brooklyn, NY 11201 USA} \email{vgitman@nylogic.org}
\title{Ramsey-like Cardinals}
\maketitle
\begin{abstract}
One of the numerous characterizations of a Ramsey cardinal $\kappa$
involves the existence of certain types of elementary embeddings for
transitive sets of size $\kappa$ satisfying a large fragment of {\rm
ZFC}. We introduce new large cardinal axioms generalizing the Ramsey
elementary embeddings characterization and show that they form a
natural hierarchy between weakly compact cardinals and measurable
cardinals. These new axioms serve to further our knowledge about the
elementary embedding properties of smaller large cardinals, in
particular those still consistent with $V=L$.
\end{abstract}
\section{Introduction}
Most large cardinals, including measurable cardinals and stronger
notions, are defined in terms of the existence of elementary
embeddings with that cardinal as the critical point. Several smaller
large cardinals, such as weakly compact, indescribable, and Ramsey
cardinals, have definitions in terms of elementary embeddings, but
are more widely known for their other properties (combinatorial,
reflecting, etc.). Other smaller large cardinals, such as ineffable
and subtle cardinals, have no known elementary embedding
characterizations. We will investigate the elementary embedding
characterization of Ramsey cardinals and introduce new large
cardinal axioms by generalizing the Ramsey embeddings. By placing
these new large cardinals within the existing hierarchy, we try to
shed light on the variety of elementary embedding properties that
are possible for smaller large cardinals. In a forthcoming paper
with Thomas Johnstone \cite{gitman:ramseyindes}, we use the new
large cardinals to obtain some basic indestructibility results for
Ramsey cardinals through the techniques of lifting embeddings. It is
hoped that this project will motivate set theorists who work with
smaller large cardinals to focus on investigating their elementary
embedding properties.

Smaller large cardinals usually imply the existence of
embeddings\footnote{Throughout the paper, unless specifically stated
otherwise, the sources and targets of elementary embeddings are
assumed to be \emph{transitive} sets or classes.} for
``mini-universes" of set theory of size $\kappa$ and height above
$\kappa$ that we call \emph{weak $\kappa$-models} and
\emph{$\kappa$-models} of set theory. Let ${\rm ZFC}^-$ denote {\rm
ZFC} without the Powerset axiom. A transitive set $M\models {\rm
ZFC}^-$ of size $\kappa$ with $\kappa\in M$  is a \emph{weak
$\kappa$-model} of set theory. A weak $\kappa$-model $M$ is a
\emph{$\kappa$-model} if additionally $M^{<\kappa}\subseteq M$.
Observe that for any cardinal $\kappa$, if $M\prec\her{\kappa}$ has
size $\kappa$ with $\kappa\subseteq M$, then $M$ is a weak
$\kappa$-model. Similarly, for regular $\lambda
>\kappa$, if $X\prec H_\lambda$ has size $\kappa$ with
$\kappa+1\subseteq X$, then the Mostowski collapse of $X$ is a weak
$\kappa$-model. So there are always many weak $\kappa$-models for
any cardinal $\kappa$. If additionally $\kappa^{<\kappa}=\kappa$, we
can use a L\"owenheim-Skolem type construction to build
$\kappa$-models $M\prec \her{\kappa}$ and substructures $X\prec
H_\lambda$ whose collapse will be a $\kappa$-model.

To provide a motivation for the new large cardinal notions we are
introducing, let us recall the various equivalent characterizations
of weakly compact cardinals.\footnote{The proofs of these
equivalences are standard and parts of them can be found in
\cite{kanamori:higher} (Ch 1, Sec. 4 and Ch. 2, Sec 7.) and
\cite{cummings:weaklycompact} (Sec. 16).}
\begin{theorem}\label{th:wc}
If $\kappa^{<\kappa}=\kappa$, then the following are equivalent:
\begin{itemize}
\item[(1)] \emph{(Compactness Property)} $\kappa$ is weakly compact. That is, $\kappa$ is
uncountable and every $<\kappa$-satisfiable theory in a
$L_{\kappa,\kappa}$ language of size at most $\kappa$ is
satisfiable.
\item[(2)] \emph{(Extension Property)} For every $A\subseteq\kappa$,
there is a transitive structure $W$ properly extending $V_\kappa$
and $A^*\subseteq W$ such that $\la V_\kappa, \in,A\ra\prec \la
W,\in, A^*\ra$.
\item[(3)] \emph{(Tree Property)} $\kappa$ is inaccessible and every $\kappa$-tree has a cofinal branch.
\item[(4)] \emph{(Embedding Property)} Every $A\subseteq\kappa$ is contained in weak
$\kappa$-model $M$ for which there exists an elementary embedding
$j:M\to N$ with critical point $\kappa$.
\item[(5)] Every $A\subseteq\kappa$ is contained in a $\kappa$-model
$M$ for which there exists an elementary embedding $j:M\to N$ with
critical point $\kappa$.
\item[(6)] Every $A\subseteq\kappa$ is contained in a $\kappa$-model $M\prec
\her{\kappa}$ for which there exists an elementary embedding $j:M\to
N$ with critical point $\kappa$.
\item[(7)] For every $\kappa$-model $M$, there exists an elementary embedding
$j:M\to N$ with critical point $\kappa$.
\item[(8)] \emph{(Hauser Property)} For every $\kappa$-model $M$, there exists an elementary embedding
$j:M\to N$ with critical point $\kappa$ such that $j$ and $M$ are
elements of $N$.
\end{itemize}
\end{theorem}

Call an elementary embedding $j:M\to N$ between transitive models of
$\rm{ZFC}^-$ $\kappa$-\emph{powerset preserving} if it has critical
point $\kappa$ and $M$ and $N$ have the \emph{same} subsets of
$\kappa$. Note that if $j:M\to N$ has critical point $\kappa$, then
$\Power(\kappa)^M\subseteq \Power(\kappa)^N$, and so for such an
embedding to be $\kappa$-powerset preserving, $N$ must not acquire
additional subsets of $\kappa$. For example, it is trivially true
that any elementary embedding $j:V\to M$ with critical point
$\kappa$ is $\kappa$-powerset preserving since $M\subseteq V$. Many
common set theoretic constructions involve building a directed
system of elementary embeddings by iterating the ultrapower
construction starting from a measure. The existence of the measure
required for the second step of this iteration is equivalent to the
ultrapower embedding by the initial measure being $\kappa$-powerset
preserving. This makes $\kappa$-powerset preservation a necessary
precondition for iterating the ultrapower construction (see Sections
\ref{sec:prelim} and \ref{sec:ramsey}), and thus a natural notion to
study. Another important motivation for focusing on
$\kappa$-powerset preserving embeddings comes from a general trend
in the theory of large cardinals of making the source and target of
the embedding closely related to derive various reflection
properties.

The general idea is to consider the elementary embedding
characterizations of weakly compact cardinals from Theorem
\ref{th:wc} (4)-(7) with the added assumption that the embeddings
have to be \emph{$\kappa$-powerset preserving}. We will soon see
that this innocuous looking assumption destroys the equivalence in
the strongest possible sense. We are now ready to introduce the new
\emph{Ramsey-like} large cardinal notions: \emph{weakly Ramsey}
cardinals, \emph{strongly Ramsey} cardinals, \emph{super Ramsey}
cardinals, and \emph{superlatively Ramsey} cardinals. These and
related large cardinal notions will be the subject of this paper.
\begin{definition}\label{def:wrp}
A cardinal $\kappa$ is \emph{weakly Ramsey} if every
$A\subseteq\kappa$ is contained in a weak $\kappa$-model $M$ for
which there exists a $\kappa$-powerset preserving elementary
embedding $j:M\to N$.
\end{definition}
Recall that a cardinal $\kappa$ is \emph{Ramsey} if every coloring
$f:[\kappa]^{<\omega}\to 2$ has a homogeneous set of size $\kappa$.
The connection between the new notions and Ramsey cardinals is seen
in the following theorem implicit in \cite{dodd:coremodel} and
\cite{mitchell:ramsey}.
\begin{theorem}
\label{th:ramseycard} A cardinal $\kappa$ is Ramsey if and only if
every $A\subseteq\kappa$ is contained in a weak $\kappa$-model $M$
for which there exists a $\kappa$-powerset preserving elementary
embedding $j:M\to N$ with the additional property that whenever $\la
A_n\mid n\in \omega\ra$ is a sequence of subsets of $\kappa$ such
that each $A_n\in M$ and $\kappa\in j(A_n)$, then $\cap_{n\in\omega}
A_n \neq \emptyset$.
\end{theorem}
 For $\la A_n\mid n\in\omega\ra\in M$, the conclusion
follows trivially by elementarity, so the content here is for
sequences \emph{not} in M. A proof of Theorem \ref{th:ramseycard}
will be sketched in Sections \ref{sec:hierarchy} and
\ref{sec:ramsey}. We will see later that the measure derived from an
embedding of Theorem \ref{th:ramseycard} allows the ultrapower
construction to be iterated through all the ordinals (see Section
\ref{sec:iterable}). Thus, there is an obvious gap between weakly
Ramsey cardinals, where we can only begin iterating by constructing
the second measure, and Ramsey cardinals, where we can already
iterate the construction through all the ordinals. This gap will be
filled by a refined hierarchy of new large cardinal notions which
are introduced in Section \ref{sec:iterable} and form the subject of
\cite{gitman:welch}.
\begin{definition}\label{def:srp}
A cardinal $\kappa$ is \emph{strongly Ramsey} if every
$A\subseteq\kappa$ is contained in a $\kappa$-model $M$ for which
there exists a $\kappa$-powerset preserving elementary embedding
$j:M\to N$.
\end{definition}
A motivation for introducing strongly Ramsey cardinals lies in the
techniques used to demonstrate the indestructibility of large
cardinals by certain kinds of forcing. Having $M^{<\kappa}\subseteq
M$ makes it possible to use the standard techniques of lifting the
embedding to a forcing extension. Note that strongly Ramsey
cardinals are clearly Ramsey since every $\la A_n\mid
n\in\omega\ra\subseteq M$ is an element of $M$.
\begin{definition}\label{def:strongrp}
A cardinal $\kappa$ is \emph{super Ramsey} if every
$A\subseteq\kappa$ is contained in a $\kappa$-model $M\prec
\her{\kappa}$ for which there exists a $\kappa$-powerset preserving
elementary embedding $j:M\to N$.
\end{definition}
A motivation for introducing super Ramsey cardinals is that having
$M\prec \her{\kappa}$ guarantees that $M$ is stationarily correct.
\begin{definition}\label{def:trp}
A cardinal $\kappa$ is \emph{superlatively Ramsey} if for every
$\kappa$-model\break $M\prec \her{\kappa}$, there exists a
$\kappa$-powerset preserving elementary embedding $j:M\to N$.
\end{definition}

Since a weak $\kappa$-model can take the Mostowski collapse of any
of its elements, Definitions \ref{def:wrp}, \ref{def:srp},
\ref{def:strongrp}, \ref{def:trp} and Theorem \ref{th:ramseycard}
hold not just for any $A\subseteq\kappa$, but more generally for any
$A\in\her{\kappa}$.

The following theorem summarizes what is known about where the new
large cardinals fit into the existing hierarchy. Also, see the
diagram that follows.
\begin{theorem}\label{th:main}
$\,$
\begin{itemize}
\item[(1)] A measurable cardinal is a super Ramsey limit of super Ramsey cardinals.
\item[(2)] A super Ramsey cardinal is a strongly Ramsey limit of strongly
Ramsey cardinals.
\item[(3)] A strongly Ramsey cardinal is a limit of completely Ramsey cardinals. It is
Ramsey but not necessarily completely
Ramsey.\footnote{\emph{Completely Ramsey} cardinals top a hierarchy
of large cardinals generalizing Ramsey cardinals and were introduced
in \cite{feng:ramsey}.}
\item[(4)] A Ramsey cardinal is a weakly Ramsey limit of weakly
Ramsey cardinals.
\item[(5)] Weakly Ramsey cardinals are consistent with
$V=L$.
\item[(6)] A weakly Ramsey cardinal is a weakly ineffable
limit of completely ineffable cardinals.\footnote{\emph{Completely
ineffable} cardinals were introduced in \cite{kleinberg:ineffable}.}
\item[(7)] There are no superlatively Ramsey cardinals.
\end{itemize}
\end{theorem}
The proofs of all statements excluding (4) and (5) are given in
Section \ref{sec:hierarchy}. Statements (4) and (5) are proved in
the upcoming paper \cite{gitman:welch}. Theorem \ref{th:main} shows,
surprisingly, that the various embedding characterizations of weakly
compact cardinals (Theorem \ref{th:wc}) form a hierarchy of strength
when the powerset preservation property is added. The most general
such property is even inconsistent! This hierarchy fits quite
naturally into the large cardinal hierarchy and suggests further
refinements such as those introduced in Section \ref{sec:iterable}.

The diagram on the next page illustrates how the new notions fit
into the existing hierarchy. The solid arrows indicate direct
implications and the dashed arrows indicate consistency strength.
\newpage

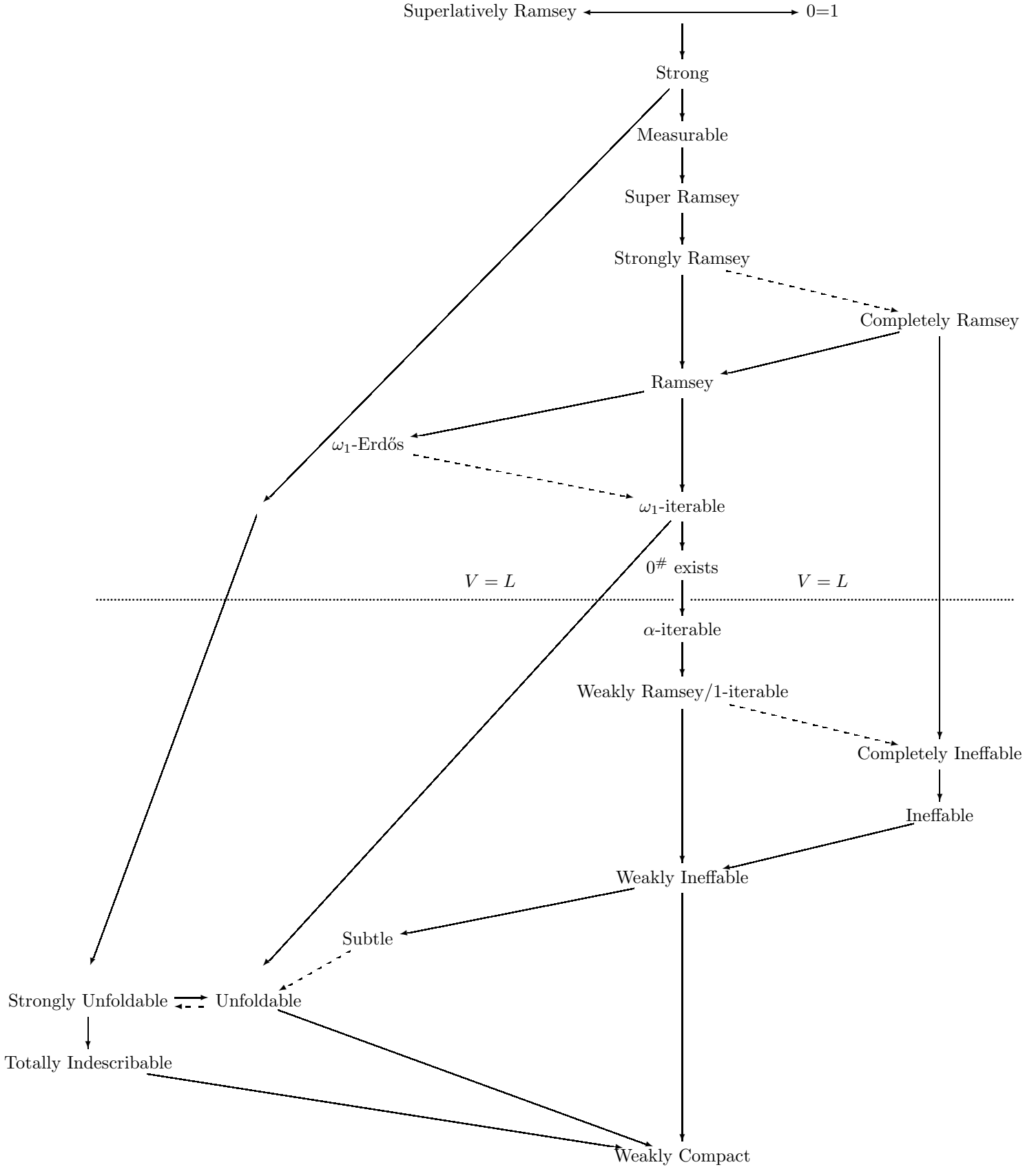
\begin{diagram}[height=1.7em]
&  &                 &          &                &\text{Superlatively Ramsey}&\lCorresponds             &\text{0=1}&                          \\
&  &                 &          &                &                        &\dTo                      &         &                          \\
&  &                 &          &                &                        &\text{Strong}          &         &                          \\
&  &                 &          &                &\ldTo(4,14)             &\dTo                      &         &                          \\
&  &                 &          &                &                        &\text{Measurable}      &         &                          \\
&  &                 &          &                &                        &\dTo                      &         &                          \\
&  &                 &          &                &                        &\text{Super Ramsey}    &         &                          \\
&  &                 &          &                &                        &\dTo                      &    &                          \\
&  &                 &          &                &                        &\text{Strongly Ramsey} &        & \\
&  &                 &          &                &                        &                       & \rdDashto                &    \\
&  &                 &                 &          &                &                        &                       & \text{Completely Ramsey}    \\
&  &                 &          &                &                        &\dTo                      &\ldTo         &  \dTo                   \\
&  &                 &          &                &                        &\text{Ramsey}          &         &                          \\
&  &                 &          &                & \ldTo                  &                       &         &           \\
&  &                 &           &\text{$\omega_1$-Erd\H{o}s}&   &                       &           &\\
&  &                 &          &                & \rdDashto              &\dTo                      &         &                          \\
&  &                 &          &                &                        &\text{$\omega_1$-iterable} &         &                          \\
&\ldTo(2,15)&        &          &                &\ldTo(4,15)             &\dTo                                      &         &        \\
&  &                 &          &                &                        &\text{$0^\#$ exists}                     &                  &                          \\
&   &        &          &                &\rDotsto^{V=L}                                &\dTo                                    &\lDotsto^{V=L}        &        \\
&  &                 &          &                &                        &\text{$\alpha$-iterable}                     &                  &                          \\
&  &                 &          &                &                        & \dTo                   &                  &                          \\
&  &                 &          &                &                        &\text{Weakly Ramsey/1-iterable}   &         &               \\
&  &                 &          &                &                        &\dTo                      &\rdDashto&                          \\
&  &                 &          &                &                        &                          &         &\text{Completely Ineffable}       \\
&  &                 &          &                &                        &                          &         &\dTo                          \\
&  &                 &          &                &                        &                         &         & \text{Ineffable}               \\
&  &                 &          &                &                        &                         &  \ldTo        &                        \\
&  &                 &          &                &                        &\text{Weakly Ineffable}&         &                          \\
&  &                 &          &                &\ldTo                   &\dTo                      &         &                          \\
&  &                 &          &\text{Subtle}&                        &                          &         &                          \\
&  &                 &\ldDashto &                &                        &                          &         &                          \\
\text{Strongly Unfoldable}&\pile{\rTo\\\lDashto}&\text{Unfoldable}&  &    &  &                 &         &                          \\
\dTo&  &             &\rdTo(4,5)          &                &                        &                          &         &                          \\
\text{Totally Indescribable}&  &                 &          &                &         &                          &         &                          \\
&\rdTo(6,3)  &                 &          &                &                   &                          &         &                          \\
&  &                 &          &                &                        &                          &         &                          \\
&  &                 &          &                &                        &\text{Weakly Compact}  &         &                          \\
\end{diagram}
\newpage
\section{Ramsey-like embeddings and ultrafilters}\label{sec:prelim}
The existence of an elementary embedding $j:V\to M$ with critical
point $\kappa$ is equivalent to the existence of a $\kappa$-complete
ultrafilter\footnote{Throughout the paper, filters are assumed to be
nonprincipal.} on $\kappa$. Similarly the existence of such an
elementary embedding for a weak $\kappa$-model of set theory is
equivalent to the existence of a certain filter measuring all the
subsets of $\kappa$ of the model. In this section, we will recall
the various properties of such ``mini" ultrafilters and use them to
characterize the new Ramsey-like large cardinal notions.
\begin{definition}
Suppose $M$ is a transitive model of $\rm{ZFC}^-$ and $\kappa$ is a
cardinal in $M$. A set $U\subseteq \Power(\kappa)^M$ is an
$M$-\emph{ultrafilter} if $\la M,\in,U\ra\models ``U$ is a
$\kappa$-complete normal ultrafilter".
\end{definition}
Recall that an ultrafilter is $\kappa$-\emph{complete} if the
intersection of any $<\kappa$-sized collection of sets in the
ultrafilter is itself an element of the ultrafilter. An ultrafilter
is \emph{normal} if every function regressive on a set in the
ultrafilter is constant on a set in the ultrafilter, or equivalently
if it is closed under diagonal intersections of length $\kappa$. By
definition, $M$-ultrafilters are $\kappa$-complete and normal only
from the \emph{point of view} of $M$, that is, the collection of
sets being intersected or diagonally intersected has to be an
element of $M$. We will say that an $M$-ultrafilter is
\emph{countably complete} if every countable collection of sets in
the ultrafilter has a nonempty intersection. Obviously, any
$M$-ultrafilter is, by definition, countably complete from the point
of view of $M$, but countable completeness requires the property to
hold of \emph{all} sequences, not just those in $M$.\footnote{It is
more standard for countable completeness to mean
$\omega_1$-completeness which requires the intersection to be an
element of the ultrafilter. However, the weaker notion we use here
is better suited to $M$-ultrafilters because the countable
collection itself can be external to $M$, and so there is no reason
to suppose the intersection to be an element of $M$.} Unless $M$
satisfies some extra condition, such as being closed under countable
sequences, an $M$-ultrafilter need not be countably complete. A
modified version of the \L o\'{s} ultrapower construction using only
functions on $\kappa$ that are elements of $M$ can be carried out
with an $M$-ultrafilter. The ultrapower of $V$ by an ultrafilter on
$\kappa$ is well-founded if and only if the ultrafilter is countably
complete. For $M$-ultrafilters, there is no such nice
characterization of well-founded ultrapowers. While countable
completeness is sufficient, it is not necessary.\footnote{See
\cite{kanamori:higher} (Ch. 4, Sec. 19) for details on ultrapowers
by $M$-ultrafilters.}
\begin{definition}
Suppose $M$ is a weak $\kappa$-model. An $M$-ultrafilter $U$ on
$\kappa$ is $0$-\emph{good} if the ultrapower of $M$ by $U$ is
well-founded.
\end{definition}
The notion of 0-good $M$-ultrafilters anticipates the discussion in
Section \ref{sec:iterable}, where we introduce the generalized
notion of $\alpha$-\emph{good} $M$-ultrafilters. A 0-good
$M$-ultrafilter on $\kappa$ gives rise to an elementary embedding of
$M$ with critical point $\kappa$, that is, the ultrapower embedding.
The next proposition shows conversely that an elementary embedding
of $M$ with critical point $\kappa$ gives rise to a 0-good
$M$-ultrafilter on $\kappa$.
\begin{proposition}\label{prop:wlog} Suppose $M$ is a weak
$\kappa$-model and $j:M\to N$ is an elementary embedding with
critical point $\kappa$. If we let $X=\{j(f)(\kappa)\mid\
f:\kappa\to M\text{ and }f\in M\}$ and $\pi:X\to K$ be the Mostowski
collapse, then $h=\pi\circ j:M\to K$ is an elementary embedding with
critical point $\kappa$ having the following properties:
\begin{itemize}
\item[(1)] $h$ is the ultrapower map by an $M$-ultrafilter on
$\kappa$,
\item[(2)] the target $K$ has size $\kappa$,
\item[(3)] the following diagram is commutative:
\begin{diagram}
M & & & &\\
\dTo^h & \rdTo(4,2)^j & & &\\
K & &\rTo^{\pi^{-1}} & & N
\end{diagram}
\item[(4)] if $j$ was $\kappa$-powerset preserving, then so is $h$,
\item[(5)] if $M^\alpha\subseteq M$ for some ordinal $\alpha$, then $K^{\alpha}\subseteq K$,
\item[(6)] $\kappa\in j(A)$ if and only if $\kappa\in h(A)$ for all
$A\in \Power(\kappa)^M$.
\end{itemize}
\end{proposition}
\begin{proof}
Properties (2) and (3) follow directly from the definition of $h$.
Define $U=\{A\in \Power(\kappa)^M\mid\kappa\in j(A)\}$. It follows
by standard arguments that $U$ is an $M$-ultrafilter, the embedding
$h$ is the ultrapower  by $U$, and if $M^\alpha\subseteq M$, then
$K^\alpha\subseteq K$. This gives properties (1) and (5). For (4),
observe that $\pi(\beta)=\beta$ for all $\beta\leq \kappa$, and
hence the critical point of $\pi^{-1}$ is above $\kappa$. This
implies that $N$ and $K$ have the same subsets of $\kappa$. Finally,
for (6), we have
\begin{displaymath}
\kappa\in j(A)\leftrightarrow \pi(\kappa)\in \pi\circ
j(A)\leftrightarrow \kappa\in \pi\circ j(A)\leftrightarrow \kappa
\in h(A).
\end{displaymath}
\end{proof}
By Proposition \ref{prop:wlog}, we can assume, without loss of
generality, in Definitions \ref{def:wrp}, \ref{def:srp},
\ref{def:strongrp}, \ref{def:trp}, and Theorem \ref{th:ramseycard}
that $j:M\to N$ is the \emph{ultrapower map} by an $M$-ultrafilter
on $\kappa$, the target $N$ has size $\kappa$, and if
$M^\alpha\subseteq M$ for some ordinal $\alpha$, then
$N^\alpha\subseteq N$.

\begin{definition}\label{d:wa}
Suppose $M$ is a weak $\kappa$-model. An $M$-ultrafilter $U$ on
$\kappa$ is \emph{weakly amenable} if for every $A\in M$ of size
$\kappa$ in $M$, the intersection $U\cap A$ is an element of $M$.
\end{definition}
Equivalently, $U$ is weakly amenable if for every sequence $\la
B_\alpha\mid \alpha<\kappa\ra$ in $M$ of subsets of $\kappa$, the
set $\{\alpha<\kappa\mid B_\alpha\in U\}$ is an element of $M$.

It follows directly from the definition that if $U$ is weakly
amenable, then for every $A\subseteq\kappa^n\times\kappa$ in $M$,
the set $\{\arr{\alpha}\in \kappa^n\mid A_{\arr{\alpha}}\in U\}\in
M$.\footnote{Whenever $A\subseteq X\times Y$, we will use the
notation $A_a$ to denote the set $\{y\in Y\mid (a,y)\in A\}$. } This
allows us to define \emph{product} ultrafilters $U^n$ on
$\Power(\kappa^n)^ M$ for every $n\in\omega$. We define $U^n$ by
induction on $n$ by $A\subseteq\kappa^n\times\kappa$ is in
$U^{n+1}=U^n\times U$ if $A\in M$ and $\{\arr{\alpha}\in\kappa^n\mid
A_{\arr{\alpha}}\in U\}\in U^n$. Note that weak amenability is
clearly a prerequisite for defining product ultrafilters. The next
proposition is a standard fact about product ultrafilters that will
be used later.
\begin{proposition}\label{prop:diagonal}
Suppose $M$ is a weak $\kappa$-model and $U$ is a weakly amenable
$M$-ultrafilter on $\kappa$. Then for every  $A\in U^n$, there is
$B\in U$ such that if $\alpha_1<\ldots<\alpha_n$ are in $B$, then
$\la \alpha_1,\ldots,\alpha_n\ra\in A$.
\end{proposition}
\begin{proof}
We will argue by induction on $n\in\omega$. For the base case $n=2$,
fix $A\in U^2$. By definition of $U^2$, the set
$X=\{\alpha<\kappa\mid A_\alpha\in U\}\in U$. In $M$, define a
sequence $\la C_\alpha\mid \alpha<\kappa\ra$ of sets in $U$ by
$C_\alpha=A_\alpha$ for $\alpha\in X$ and $C_\alpha=\kappa$
otherwise. The diagonal intersection
$D=\diagonal_{\alpha<\kappa}C_\alpha$ is in $U$ by normality. We
claim that $B=D\cap X$ has the requisite property for $A$. If
$\alpha_1<\alpha_2\in B$, then $\alpha_1\in X$, and so $\alpha_2\in
C_{\alpha_1}=A_{\alpha_1}$. It follows that $\la
\alpha_1,\alpha_2\ra\in A$.

Now suppose inductively that the statement holds for $U^n$, and fix
$A\in U^{n+1}$. By definition of $U^{n+1}$, the set
$X=\{\arr{\alpha}\in\kappa^n\mid A_{\arr{\alpha}}\in U\}\in U^n$. By
the inductive assumption, fix $Y\in U$ such that for all
$\alpha_1<\cdots<\alpha_n$ in $Y$, the sequence $\la
\alpha_1,\ldots,\alpha_n\ra\in X$. It follows that for all
$\alpha_1<\cdots<\alpha_n$ in $Y$, the set
$A_{\la\alpha_1,\ldots,\alpha_n\ra}\in U$. In $M$, we define a
sequence $\la C_\alpha\mid \alpha<\kappa\ra$ of sets in $U$ as
follows. For $\alpha\in Y$, let $Seq^\alpha$ be the collection of
all increasing $n$-tuples of ordinals in $Y$ ending in $\alpha$.
Define $C_\alpha=\cap_{\arr{\beta}\in Seq^\alpha}A_{\arr{\beta}}$
for $\alpha\in Y$ and $C_\alpha=\kappa$ otherwise.  The diagonal
intersection $D=\diagonal_{\alpha<\kappa}C_\alpha$ is in $U$ by
normality. We claim that $B=D\cap Y$ has the requisite property for
$A$. If $\alpha_1<\cdots<\alpha_n<\alpha_{n+1}\in B$, then
$\alpha_1<\cdots<\alpha_n\in Y$ and $\alpha_{n+1}\in C_{\alpha_n}$.
Since $\alpha_n\in Y$, we have $C_{\alpha_n}=\cap_{\arr{\beta}\in
Seq^{\alpha_n}}A_{\arr{\beta}}$. So in particular, $\alpha_{n+1}\in
A_{\la \alpha_1,\ldots,\alpha_n\ra}$, and hence $\la
\alpha_1,\ldots,\alpha_n,\alpha_{n+1}\ra\in A$. This completes the
induction step and finishes the argument.
\end{proof}

The next proposition reformulates the property of $\kappa$-powerset
preservation in terms of the existence of weakly amenable
$M$-ultrafilters.
\begin{proposition}\label{p:wa}Suppose $M$ is a transitive model of $\rm{ZFC}^-$.
\begin{itemize}
\item[(1)] If $j:M\to N$ is the ultrapower by a weakly amenable
$M$-ultrafilter on $\kappa$, then $j$ is $\kappa$-powerset
preserving.
\item[(2)] If $j:M\to N$ is a $\kappa$-powerset preserving embedding,
then the $M$-ultrafilter $U=\{A\in\Power(\kappa)^M\mid \kappa\in
j(A)\}$ is weakly amenable. \end{itemize}
\end{proposition}
See \cite{kanamori:higher} (Ch. 4, Sec. 19) for proof.
\begin{definition}
Suppose $M$ is a weak $\kappa$-model. An $M$-ultrafilter on $\kappa$
is 1-\emph{good} if it is 0-good and weakly amenable.
\end{definition}
Now we are ready to characterize the Ramsey-like large cardinal
notions in terms of the existence of $M$-ultrafilters.
\begin{proposition}\label{prop:ultdef} $\,$\\
\begin{itemize}
\item [(1)] Suppose $\kappa^{<\kappa}=\kappa$, then $\kappa$ is weakly compact if and only if every
$A\subseteq\kappa$ is contained in a weak $\kappa$-model $M$ for
which there exists a $0$-good $M$-ultrafilter on $\kappa$.
\item[(2)] A cardinal $\kappa$ is weakly Ramsey if and only if every
$A\subseteq\kappa$ is contained a weak $\kappa$-model $M$ for which
there exists a $1$-good $M$-ultrafilter on $\kappa$.
\item[(3)] A cardinal $\kappa$ is Ramsey if and only if every
$A\subseteq\kappa$ is contained in a weak $\kappa$-model $M$ for
which there exists a weakly amenable countably complete
$M$-ultrafilter on $\kappa$.
\item[(4)] A cardinal $\kappa$ is strongly Ramsey if and only if every
$A\subseteq\kappa$ is contained in a $\kappa$-model $M$ for which
there exists a weakly amenable $M$-ultrafilter on $\kappa$.
\item[(5)] A cardinal $\kappa$ is super Ramsey if and only if every
$A\subseteq\kappa$ is contained in a $\kappa$-model $M\prec
\her{\kappa}$ for which there exists a weakly amenable
$M$-ultrafilter on $\kappa$.
\end{itemize}
\end{proposition}
\begin{proof}
Statement (1) follows from Theorem \ref{th:wc}(4) and Proposition
\ref{prop:wlog}. Statement (2) follows from Proposition
\ref{prop:wlog} and Proposition \ref{p:wa}. Statement (3) follows
Proposition \ref{prop:wlog} and the fact that countably complete
$M$-ultrafilters are 0-good. Statements (4) and (5) follow because
for a $\kappa$-model $M$, an $M$-ultrafilter is always countably
complete.
\end{proof}

We will say that a weak $\kappa$-model $M\models$``I am
$\her{\kappa}$" if $M$ thinks all its sets have transitive closure
of size at most $\kappa$. By replacing $M$ with $\bar{M}=\{a\in
M:M\models$ transitive closure of $a$ has size $\leq\kappa\}$, we
can assume, without loss of generality, that $M\models$``I am
$\her{\kappa}$" in every statement of Proposition \ref{prop:ultdef}.
The model $\bar{M}$ is a weak $\kappa$-model that will be as closed
under $<\kappa$-sequences as $M$ and the $M$-ultrafilter $U$ will
retain the necessary properties when viewed as an
$\bar{M}$-ultrafilter.
\section{The Large Cardinal Hierarchy}\label{sec:hierarchy}
In this section, we prove parts (1)-(3) and (6)-(7) of Theorem
\ref{th:main}. We also show the backward direction of Theorem
\ref{th:ramseycard} as restated in Proposition \ref{prop:ultdef}(3).
The arguments below rely mainly on the powerful reflecting
properties of $\kappa$-powerset preserving embeddings. It is also
important to note that if two transitive models of $\rm{ZFC}^-$ have
the same subsets of $\kappa$, then they also have the same sets with
transitive closure of size at most $\kappa$.
\begin{proposition}
Weakly Ramsey cardinals are inaccessible.
\end{proposition}
\begin{proof}
Suppose $\kappa$ is weakly Ramsey. If $\kappa$ is not regular, then
there is a cofinal map $f:\alpha\to \kappa$ for some
$\alpha<\kappa$. Choose a weak $\kappa$-model $M$ containing $f$ for
which there is an embedding $j:M\to N$ with critical point $\kappa$.
Now observe that $j(f)=f$ and $j(f)$ must be cofinal in $j(\kappa)$
by elementarity, which is a contradiction. If $\kappa$ is not a
strong limit, then there is $\alpha<\kappa$ with
$|\Power(\alpha)|\geq\kappa$. Fix
$f:\kappa\overset{\text{1-1}}{\to}\Power(\alpha)$ and choose a weak
$\kappa$-model $M$ containing $f$ for which there is a
$\kappa$-powerset preserving embedding $j:M\to N$. Let
$j(f)(\kappa)=A\subseteq\alpha$. The set $A$ is in $M$ by
$\kappa$-powerset preservation, and so $j(f)(\kappa)=A=j(A)$ since
$A\subseteq\alpha$. Thus, $N\models \exists
\xi<j(\kappa)\,j(f)(\xi)=j(A)$, and so by elementarity, $M\models
\exists\xi<\kappa\,f(\xi)=A$. Now we have $j(f)(\xi)=f(\xi)=A$ and
$j(f)(\kappa)=A$, which contradicts that $j(f)$ is 1-1 by
elementarity.
\end{proof}
Notice that $\kappa$-powerset preservation is not required to show
regularity. However, without $\kappa$-powerset preservation it would
be impossible to show that $\kappa$ is a strong limit since property
$(4)$ of weakly compact cardinals from Theorem \ref{th:wc} without
$\kappa^{<\kappa}=\kappa$ does not imply that $\kappa$ is a strong
limit. To show this we start with a weakly compact cardinal $\kappa$
and force to add $\kappa^+$ many Cohen reals. In the forcing
extension, $\kappa$ is clearly no longer a strong limit but it can
be shown that it retains property (4) from Theorem
\ref{th:wc}.\footnote{The argument is due to Hamkins and will appear
in \cite{hamkins:book} (Ch. 6).}

\begin{definition}
An uncountable regular cardinal $\kappa$ is \emph{weakly ineffable}
if for every sequence $\la A_\alpha\mid \alpha\in\kappa\ra$ with
$A_\alpha\subseteq \alpha$, there exists $A\subseteq\kappa$ such
that the set\break $S=\{\alpha\in\kappa\mid A\cap \alpha=A_\alpha\}$
has size $\kappa$. An uncountable regular cardinal $\kappa$ is
\emph{ineffable} if such $A$ can be found for which the
corresponding set $S$ is stationary.
\end{definition}
Ineffable cardinals are limits of weakly ineffable cardinals. This
follows since ineffable cardinals are $\Pi_2^1$-indescribable and
being weakly ineffable is a $\Pi_2^1$- statement satisfied by
ineffable cardinals. Ramsey cardinals are limits of ineffable
cardinals but need not be ineffable. Since being Ramsey is a
$\Pi_2^1$-statement, a Ramsey cardinal that is ineffable is a limit
of Ramsey cardinals. In particular, the least Ramsey cardinal cannot
be ineffable. Like Ramsey cardinals, ineffable cardinals can be
characterized by the existence of homogeneous sets for colorings. An
uncountable cardinal $\kappa$ is ineffable if and only if every
coloring $f:[\kappa]^2\to 2$ is homogeneous on a stationary
set.\footnote{For an introduction to ineffability see
\cite{devlin:constructibility} (Ch. 7, Sec. 2).}
\begin{theorem}\label{th:ineffable} Weakly Ramsey cardinals are
weakly ineffable.
\end{theorem}
\begin{proof}
Suppose $\kappa$ is weakly Ramsey. Fix $\overset{\to}{A}=\la
A_\alpha\mid \alpha\in\kappa\ra$ with each
$A_\alpha\subseteq\alpha$. Choose a weak $\kappa$-model $M$
containing $\overset{\to}{A}$ for which there exists a
$\kappa$-powerset preserving embedding $j:M\to N$ and let
$A=j(\overset{\to}{A})(\kappa)$. Since $A\subseteq\kappa$, by
$\kappa$-powerset preservation, $A\in M$. It is easy to see that the
set $S=\{\alpha\in\kappa\mid A\cap\alpha=A_\alpha\}$ is stationary
in $M$. Namely, fix a club $C\in M$ and observe that $\kappa\in
j(S)\cap j(C)$. In particular, $S$ has size $\kappa$.
\end{proof}
Since Ramsey cardinals are weakly Ramsey, it follows that not every
weakly Ramsey cardinal is ineffable. However, consistency
strength-wise, weakly Ramsey cardinals are stronger than ineffable
cardinals. We show below that weakly Ramsey cardinals are limits of
completely ineffable cardinals, where complete ineffability is a
strengthening of ineffability introduced by
\cite{kleinberg:ineffable}.
\begin{definition}
A collection $R\subseteq \Power(\kappa)$ is a \emph{stationary
class} if
\begin{itemize}
\item [(1)] $R\neq\emptyset$,
\item[(2)] for all $A\in R$, $A$ is stationary in $\kappa$,
\item[(3)] if $A\in R$ and $B\supseteq A$, then $B\in R$.
\end{itemize}
\end{definition}
\begin{definition}\label{def:completelyineffable}
A cardinal $\kappa$ is \emph{completely ineffable} if there is a
stationary class $R$ such that for every $A\in R$ and $f:[A]^2\to
2$, there is $H\in R$ homogeneous for $f$.
\end{definition}
In particular, since $\kappa\in R$ for every stationary class $R$,
it follows that if $\kappa$ is completely ineffable, then every
$f:[\kappa]^2\to 2$ has a homogeneous set that is stationary in
$\kappa$. Thus, completely ineffable cardinals are clearly
ineffable.
\begin{lemma}\label{le:homsets}
If $M$ is a weak $\kappa$-model, $U$ is a $1$-good $M$-ultrafilter
on $\kappa$, $A\in U$, and $f:[A]^{<\omega}\to 2$ is in $M$, then
for every $n\in\omega$, there is $H_n\in U$ homogeneous for
$f\restrict[A]^n$.
\end{lemma}
\begin{proof}
Recall that 1-good ultrafilters are weakly amenable, and so we can
define the product ultrafilters $U^n$ for every $n\in\omega$. It is
easy to see that either $\{\arr{\alpha}\in[A]^n\mid
f(\arr{\alpha})=0\}\in U^n$ or $\{\arr{\alpha}\in[A]^n\mid
f(\arr{\alpha})=1\}\in U^n$. Call $X$ the one of the above sets
which is in $U^n$. By Proposition \ref{prop:diagonal}, there is
$B\in U$ such that for all $\alpha_1<\cdots<\alpha_n$ in $B$, the
sequence $\la \alpha_1,\ldots,\alpha_n\ra\in X$. Clearly, letting
$B=H_n$ works.
\end{proof}

\begin{theorem}\label{th:compineff}
Weakly Ramsey cardinals are limits of completely ineffable
cardinals.
\end{theorem}
\begin{proof}
Suppose $\kappa$ is weakly Ramsey. Using Proposition
\ref{prop:ultdef}(2), fix a weak $\kappa$-model $M$ containing
$V_\kappa$ for which there exists a 1-good $M$-ultrafilter $U$ on
$\kappa$, and let $j:M\to N$ be the $\kappa$-powerset preserving
ultrapower by $U$. We will argue that $N\models ``\kappa$ is
completely ineffable". Thus, for every $\alpha<\kappa$, the model
$N$ will satisfy that there is a completely ineffable cardinal
between $\alpha$ and $j(\kappa)$. By elementarity, $M$ will satisfy
that there is a completely ineffable cardinal between $\alpha$ and
$\kappa$. Since $M$ contains $V_\kappa$, it will be correct about
this.

To show that $\kappa$ is completely ineffable in $N$, we need to
build in $N$ a stationary class $R$ satisfying the property of
Definition \ref{def:completelyineffable}. For the argument that
follows it is crucial that $\Power(\kappa)$ exists in $N$. Working
inside $N$, we will define a sequence $\la R_\alpha\mid \alpha\in
Ord^N\ra$ of subsets of $\Power(\kappa)^N$ as follows. Define $R_0$
to be the collection of all stationary subsets of $\kappa$.
Inductively, given $R_\alpha$, define $R_{\alpha+1}$ to be the set
of all $A\in R_\alpha$ such that for every $f:[A]^2\to 2$, there is
$H\in R_\alpha$ homogeneous for $f$. At limits take intersections.
Since the $R_\alpha$ form a decreasing sequence, there is $\theta$
such that $R_\theta=R_{\theta+1}$. Letting $R=R_\theta$, we argue
that it is a stationary class. Note that $R$ satisfies the property
of Definition \ref{def:completelyineffable} by construction. It is
clear that $R$ consists of stationary sets and is closed under
supersets. We verify that $R$ is nonempty by showing that
$U\subseteq R$. Since all elements of $U$ are stationary in $N$, it
follows that $U\subseteq R_0$. Inductively suppose $U\subseteq
R_\alpha$. Fix $A\in U$ and $f:[A]^2\to 2$ in $N$. By
$\kappa$-powerset preservation, $f\in M$, and so by Lemma
\ref{le:homsets} there is $H\in U$ homogeneous for $f$, making $A\in
R_{\alpha+1}$. This completes the argument that $U\subseteq R$, and
hence $R\neq\emptyset$.
\end{proof}
Theorems \ref{th:ineffable} and \ref{th:compineff} establish Theorem
\ref{th:main}(6).
\begin{proposition}
If $\kappa$ is a weakly Ramsey cardinal, then $\diamondsuit_\kappa$
holds.
\end{proposition}
\begin{proof}
A weakly Ramsey cardinal is weakly ineffable and
$\diamondsuit_\kappa$ holds for weakly ineffable
cardinals.\footnote{It is shown in \cite{devlin:constructibility}
(Ch. 7, Sec. 2) that $\diamondsuit_\kappa$ holds for ineffable
$\kappa$ and the proof uses only weak ineffability.}
\end{proof}
\begin{theorem}
Ramsey cardinals are limits of weakly Ramsey cardinals.
\end{theorem}
The theorem is included here for completeness of presentation. The
proof relies on the properties of the refined hierarchy of
Ramsey-like large cardinal notions introduced in Section 5. The
hierarchy starts with weakly Ramsey cardinals, has length
$\omega_1+1$ and is bounded above by Ramsey cardinals. In
\cite{gitman:welch}, we show that every large cardinal in the
hierarchy is a limit of each of the large cardinals in the hierarchy
below it. The theorem follows immediately from this fact.

The next theorem establishes the backward direction of Proposition
\ref{prop:ultdef}(3), and hence of Theorem \ref{th:ramseycard}. The
forward direction is a much more complicated argument and will be
discussed in detail in Section \ref{sec:ramsey}.
\begin{theorem}\label{th:backwardram}
If every $A\subseteq\kappa$ is contained in a weak $\kappa$-model
$M$ for which there exists a weakly amenable countably complete
$M$-ultrafilter on $\kappa$, then $\kappa$ is Ramsey.
\end{theorem}
\begin{proof}
Fix $f:[\kappa]^{<\omega}\to 2$ and put it into a weak
$\kappa$-model $M$ for which there exists a weakly amenable
countably complete $M$-ultrafilter $U$ on $\kappa$. Recall that, by
definition of countable completeness, every countable sequence of
sets in $U$ has a nonempty intersection. In fact, the intersection
must have size $\kappa$. To see this, suppose to the contrary that
$A_n$ for $n\in\omega$ are elements of $U$ and $A=\cap_{n\in\omega}
A_n$ is bounded by some $\alpha<\kappa$. Since $\kappa\setminus
\alpha$ is an element of $U$, the intersection of all $A_n$ together
with $\kappa\setminus\alpha$ must be nonempty by countable
completeness of $U$, but this is obviously false. Thus, indeed, the
intersection must have size $\kappa$. By Lemma \ref{le:homsets}, for
every $n\in\omega$, there is $H_n\in U$ homogeneous for
$f\restrict[\kappa]^n$. The intersection $H=\cap_{n\in\omega} H_n$
is a set of size $\kappa$ homogeneous for $f$.
\end{proof}

\begin{theorem}
Strongly Ramsey cardinals are limits of Ramsey cardinals.
\end{theorem}
\begin{proof}
Suppose $\kappa$ is strongly Ramsey. Using Proposition
\ref{prop:ultdef}, fix a $\kappa$-model $M$ for which there is a
weakly amenable $M$-ultrafilter $U$ on $\kappa$ and let $j:M\to N$
be the $\kappa$-powerset preserving ultrapower by $U$. Recall that
since $M$ is a $\kappa$-model, $U$ must be countably complete. Also,
$\kappa$-models always contain $V_\kappa$ as an element. We will
argue that $N\models ``\kappa$ is Ramsey". As before, this suffices
to show that $\kappa$ is a limit of Ramsey cardinals. Fix
$f:[\kappa]^{<\omega}\to 2$ in $N$ and observe that it is in $M$ by
$\kappa$-powerset preservation. By Lemma \ref{le:homsets}, for every
$n\in\omega$, there is $H_n\in U$ homogeneous for
$f\restrict[\kappa]^n$. Since $M^{<\kappa}\subseteq M$, the sequence
$\la H_n:n\in\omega\ra$ is an element of $M$, and hence
$H=\cap_{n\in\omega}H_n$ is an element of $M$. Clearly $H$ is an
element of $N$ as well.
\end{proof}
Note that for the proof above it clearly suffices that
$M^\omega\subseteq M$. Thus, if every $A\subseteq\kappa$ is
contained in a weak $\kappa$-model $M$ with $M^{\omega}\subseteq M$
for which there exists a $\kappa$-powerset preserving embedding,
then $\kappa$ is a limit of Ramsey cardinals. It is an interesting
open question whether it suffices to assume that
$\tmop{cf}^V(\kappa^+)^M\geq\omega_1$ (if $\kappa$ is the largest
cardinal in $M$, then $(\kappa^+)^M=Ord^M$).
\begin{question}
Does the following large cardinal property have more strength than a
Ramsey cardinal: every $A\subseteq\kappa$ is contained in a weak
$\kappa$-model $M$ with $\tmop{cf}^V(\kappa^+)^M\geq\omega_1$ for
which there exists a $\kappa$-powerset preserving embedding $j:M\to
N$?
\end{question}
If $2^\kappa=\kappa^+$ in $N$, then the answer is affirmative. Under
this assumption, we can show that $\kappa$ is Ramsey in $N$. First,
by Proposition \ref{prop:wlog}, we can assume that $j:M\to N$ is the
ultrapower by a 1-good $M$-ultrafilter $U$ on $\kappa$ without
losing $2^\kappa=\kappa^+$ in the target model. Fixing
$f:[\kappa]^{<\omega}\to 2$ in $N$, we know as before that for every
$n\in\omega$, there is a set $H_n$ in $U$ homogeneous for
$f\restrict[\kappa]^n$. Since $2^\kappa=\kappa^+$ holds in $N$, we
can define in $N$ an elementary chain of length $\kappa^+$ of
transitive models of size $\kappa$: $X_0\prec X_1\prec\cdots\prec
X_\alpha\prec\cdots\prec \her{\kappa}^N$ whose union is
$\her{\kappa}^N$. By assumption $(\kappa^+)^N=(\kappa^+)^M$ has
uncountable cofinality, and so all $H_n$ must be contained in some
$X_\alpha$. Since each $X_\xi\in M$, the weak amenability of $U$
implies that $u=U\cap X_\alpha$ is an element of $M$. Since all
$H_n$ are contained in $u$, the model $M$ satisfies that for every
$n\in\omega$, $u$ contains a set homogeneous for
$f\restrict[\kappa]^n$. It follows that there is a sequence $\la
H'_n\mid n\in\omega\ra$ of such sets in $u$ that is an element of
$M$. The intersection $H=\cap_{n\in\omega}H'_n$ is an element of $U$
since $u\subseteq U$ and $U$ is $\kappa$-complete for sequences in
$M$ by the definition of $M$-ultrafilter. Thus, in particular, $H$
has size $\kappa$.

Next, we show that strongly Ramsey cardinals are limits of the
\emph{completely Ramsey} cardinals that top Feng's
$\Pi_\alpha$-Ramsey hierarchy \cite{feng:ramsey}. If $I$ is an ideal
containing all the non-stationary subsets of $\kappa$, let $\mathscr
R^+(I)$ be the collection of all $X\subseteq\kappa$ such that every
$f:[X]^{<\omega}\to 2$ has a homogeneous set in
$\Power(\kappa)\setminus I$. Define an operation $\mathscr R$ on
such ideals by $\mathscr R(I)=\Power(\kappa)\setminus\mathscr
R^+(I)$. Feng showed that the $\mathscr R$ operation applied to an
ideal always yields an ideal and iterated it to define a hierarchy
of ideals on $\kappa$ as follows. Let $I_0$ be the ideal of
non-stationary subsets of $\kappa$. Define $I_{\alpha+1}=\mathscr
R(I_\alpha)$ and $I_\lambda=\cup_{\alpha<\lambda}
I_\alpha$.\footnote{Feng's definition of $I_n$ for $n\in\omega$ is
slightly more subtle, but for our purposes here this simplified
version suffices.} A cardinal $\kappa$ is \emph{completely Ramsey}
if for all $\alpha$, we have $\kappa\notin I_\alpha$. Notice the
similarity here to the completely ineffable cardinals and the fact
that it easily follows that completely Ramsey cardinals are
completely ineffable. Also, completely Ramsey cardinals are clearly
Ramsey since $\kappa\notin I_1$ implies that every
$f:[\kappa]^{<\omega}\to 2$ has a homogeneous set that is stationary
in $\kappa$.

\begin{theorem}\label{th:compramsey}
Strongly Ramsey cardinals are limits of completely Ramsey cardinals.
\end{theorem}
\begin{proof}
Suppose $\kappa$ is strongly Ramsey. By Proposition
\ref{prop:ultdef}, fix a $\kappa$-model $M$ for which there is a
weakly amenable $M$-ultrafilter $U$ on $\kappa$, and let $j:M\to N$
be the $\kappa$-powerset preserving ultrapower by $U$. Now argue
that $U$ is contained in the complement of every $I_\alpha$ as
defined in $N$. Thus, $N\models ``\kappa\notin I_\alpha$ for all
$\alpha$", and so $\kappa$ is completely Ramsey in $N$.
\end{proof}
In \cite{gitman:ramseyindes}, we show that it is consistent that
there is a strongly Ramsey cardinal that is not ineffable. This
follows by showing that it is consistent to have a strongly Ramsey
$\kappa$ with a slim $\kappa$-Kurepa tree. Ineffable cardinals can
never have slim Kurepa trees. Since completely Ramsey cardinals are
in particular ineffable, this implies that a strongly Ramsey
cardinal need not be completely Ramsey.

Theorem \ref{th:compramsey} together with the remarks above
establish Theorem \ref{th:main}(3).

If $\kappa$ is strongly Ramsey but not ineffable, it will not be
possible to put every $A\subseteq\kappa$ into a stationarily correct
$\kappa$-model for which there exists a $\kappa$-powerset preserving
embedding. A $\kappa$-model for which there exists an
$\kappa$-powerset preserving embedding always believes that $\kappa$
is ineffable and if it is stationarily correct, it will be correct
about this. This motivates the definition of super Ramsey cardinals
which guarantee that we can always get a $\kappa$-model that is
stationarily correct. In particular, note that super Ramsey
cardinals are ineffable.
\begin{theorem}\label{th:super_ramsey_are_strongly_ramsey}
Super Ramsey cardinals are limits of strongly Ramsey cardinals.
\end{theorem}
\begin{proof}
Suppose $\kappa$ is super Ramsey. Choose a $\kappa$-model $M\prec
\her{\kappa}$ for which there exists a $\kappa$-powerset preserving
embedding $j:M\to N$. As usual, it suffices to show that $N\models
``\kappa$ is strongly Ramsey". By Proposition \ref{prop:wlog}, we
can assume, without loss of generality, that the strong Ramsey
embeddings have targets of size $\kappa$. It follows that
$\her{\kappa}\models ``\kappa$ is strongly Ramsey", and therefore
$M\models ``\kappa$ is strongly Ramsey" by elementarity. By
$\kappa$-powerset preservation, $N$ has no new subsets of $\kappa$,
and so it must agree that $\kappa$ is strongly Ramsey.
\end{proof}
Theorem \ref{th:super_ramsey_are_strongly_ramsey} establishes
Theorem \ref{th:main}(2).

\begin{theorem}\label{th:inconsistent}
There are no superlatively Ramsey cardinals.
\end{theorem}
\begin{proof}
Suppose that there exists a superlatively Ramsey cardinal and let
$\kappa$ be the \emph{least} superlatively Ramsey cardinal.  Choose
any $\kappa$-model $M\prec \her{\kappa}$ and a $\kappa$-powerset
preserving embedding $j:M\to N$. The strategy will be to show that
$\kappa$ is superlatively Ramsey in $N$. Observe first that
$\her{\kappa}^N=M$. Thus, to show that $\kappa$ is superlatively
Ramsey in $N$, we need to verify in $N$ that every $\kappa$-model
$m\prec M$ has a $\kappa$-powerset preserving embedding. So let
$m\in N$ be a $\kappa$-model such that $m\prec M$. Observe that
$m\in M$ and $m\prec \her{\kappa}$. By Proposition \ref{prop:wlog},
$\her{\kappa}$ contains a $\kappa$-powerset preserving embedding for
$m$. By elementarity, $M$ contains some $\kappa$-powerset preserving
embedding $h:m\to n$. Clearly $h\in N$ as well. Thus, $\kappa$ is
superlatively Ramsey in $N$. It follows that there is a
superlatively Ramsey cardinal $\alpha$ below $\kappa$. This is, of
course, impossible since we assumed that $\kappa$ was the least
superlatively Ramsey cardinal.
\end{proof}
Theorem \ref{th:inconsistent} establishes Theorem \ref{th:main}(7).

Theorem \ref{th:inconsistent} is surprising since, as was pointed
out earlier, the embedding properties described in Definitions
\ref{def:wrp}, \ref{def:srp}, \ref{def:strongrp}, and \ref{def:trp}
without $\kappa$-powerset preservation are equivalent modulo the
assumption that $\kappa^{<\kappa}=\kappa$. Once we add the powerset
condition on the embeddings, the equivalence is strongly violated.
Of course, now the question arises whether there can be any super
Ramsey cardinals.
\begin{theorem}\label{th:consistent}
If $\kappa$ is a measurable cardinal, then $\kappa$ is a super
Ramsey limit of super Ramsey cardinals.
\end{theorem}
\begin{proof}
Suppose $\kappa$ is a measurable cardinal. Fix $A\subseteq \kappa$
and a normal $\kappa$-complete ultrafilter $U$ on $\kappa$. Take the
structure $\la \her{\kappa},\in,U\ra$ and using a L\"owenheim-Skolem
type construction, find an elementary substructure $\la N,\in, U\cap
N\ra$ such that $N$ is a $\kappa$-model and $A\in N$. Since $U$ is a
weakly amenable $\her{\kappa}$-ultrafilter, it follows by
elementarity that $U\cap N$ is a weakly amenable $N$-ultrafilter.
Thus, we found a $\kappa$-model $N$ containing $A$ for which there
exists a weakly amenable $N$-ultrafilter on $\kappa$. By Proposition
\ref{prop:ultdef}, this completes the proof that measurable
cardinals are super Ramsey. To see that they are limits of super
Ramsey cardinals, observe that if $j:V\to M$ is an elementary
embedding with critical point $\kappa$, then both $N$ and $N\cap U$
from the construction above are elements of $M$ by the virtue of
having transitive closure of size $\kappa$. So $M$ agrees that
$\kappa$ is super Ramsey, and it follows that $\kappa$ is a limit of
super Ramsey cardinals in $V$.
\end{proof}
Theorem \ref{th:consistent} completes the proof of Theorem
\ref{th:main}(1).
\begin{corollary}
Con{\rm (}\/${\rm ZFC}+\exists$ measurable cardinal{\rm )}
$\Longrightarrow$ Con{\rm (}\/${\rm ZFC}+\exists$ proper class of
super Ramsey cardinals{\rm )}
\end{corollary}
It is an interesting observation that the Ramsey-like cardinals are
incompatible with the Hauser property of weakly compact cardinals
(Theorem \ref{th:wc}(8)). The Hauser property plays a key role in
many indestructibility by forcing arguments.
\begin{proposition}\label{prop:hauser}
If a cardinal $\kappa$ has the property that every
$A\subseteq\kappa$ is contained in a weak $\kappa$-model $M$ for
which there exists a $\kappa$-powerset preserving embedding $j:M\to
N$ such that $j``\Power(\kappa)^M$ is an element of $N$, then
$\kappa$ is a limit of measurable cardinals.
\end{proposition}
\begin{proof}
Fix a weak $\kappa$-model $M$ containing $V_\kappa$ for which there
exists a $\kappa$-powerset preserving embedding $j:M\to N$ such that
$X=j``\Power(\kappa)^M$ is an element of $N$. Consider the usual
$M$-ultrafilter $U=\{B\in\Power(\kappa)^M\mid\kappa\in j(B)\}$.
Since $X\in N$, we can define the set $\{C\cap\kappa\mid C\in
X\text{ and }\kappa\in C\}=\{B\subseteq\kappa\mid \kappa\in
j(B)\}=U$ in $N$. Therefore $U$ is an element of $N$. By the
$\kappa$-powerset preservation property, $U$ is also an
$N$-ultrafilter, and hence $N$ thinks that $\kappa$ is measurable.
It follows that $\kappa$ must be a limit of measurable cardinals.
\end{proof}
Thus, having $j$ as an element of $N$ pushes the large cardinal
strength beyond a measurable cardinal. For instance, if $\kappa$ is
$2^\kappa$-supercompact, then $\kappa$ will have the above property.
To see this, fix a $2^\kappa$-supercompact embedding $j:V\to M$ and
$A\subseteq\kappa$. Choose some cardinal $\lambda$ such that
$j(\lambda)=\lambda$ and $j``2^\kappa\in \her{\lambda}^M$. We first
restrict $j$ to a set embedding $j:\her{\lambda}\to\her{\lambda}^M$
and observe that $\her{\lambda}^M\subseteq \her{\lambda}$. Thus, it
makes sense to consider the structure $\la
\her{\lambda},\her{\lambda}^M,j\ra$. Take an elementary substructure
$\la K',N',h'\ra$ of  $\la \her{\lambda}, \her{\lambda}^M,j\ra$ of
size $\kappa$ such that $A\in K'$, $j``2^\kappa\in K'$,
$\kappa+1\subseteq K'$, and $K'^{<\kappa}\subseteq K'$. Let
$\pi:K'\to K$ be the Mostowski collapse, $\pi`` N'=N$, and
$h=\pi``h'$. Observe that $N$ is the Mostowski collapse of $N'$. It
is easy to check that $K$ is a $\kappa$-model containing $A$, the
map $h:K\to N$ is a $\kappa$-powerset preserving embedding, and
$h``\Power(\kappa)^K$ is an element of $N$.

Recall from Theorem \ref{th:wc}(2) that weakly compact cardinals can
be characterized by the \emph{extension property}. Finally, we will
show that weakly Ramsey cardinals also have an extension-like
property. Suppose $\x\subseteq \Power(\kappa)$. The structure $\la
V_\kappa,\in,B\ra_{B\in \x}$ will be the structure in the language
consisting of $\in$ and unary predicate symbols for every element of
$\x$ with the natural interpretation.
\begin{theorem}
A cardinal $\kappa$ is weakly Ramsey if and only if every
$A\subseteq\kappa$ belongs to a collection $\x\subseteq
\Power(\kappa)$ such that the structure $\la V_\kappa,\in,
B\ra_{B\in\x}$ has a proper transitive elementary extension $\la
W,\in,B^*\ra_{B\in\x}$ with $\Power(\kappa)^W=\x$.
\end{theorem}
\begin{proof}
($\Longrightarrow$): Suppose that $\kappa$ is  weakly Ramsey and
$A\subseteq\kappa$. Fix a weak $\kappa$-model $M$ containing $A$ and
$V_\kappa$ for which there exists a $\kappa$-powerset preserving
embedding $j:M\to N$. If we let $\x=\Power(\kappa)^M$, then $\la
V_\kappa,\in, B\ra_{B\in\x}\prec \la
V_{j(\kappa)},\in,j(B)\ra_{B\in\x}$.

($\Longleftarrow$): Fix $A\subseteq\kappa$. The set $A$ belongs to a
collection $\x\subseteq \Power(\kappa)$ such that the structure $\la
V_\kappa,\in,B\ra_{B\in\x}$ has a proper transitive elementary
extension $\la W,\in, B^*\ra_{B\in\x}$ with $\Power(\kappa)^W=\x$.
We can assume that $W$ has size $\kappa$ since if this is not the
case, we can take an elementary substructure of size $\kappa$ which
contains $V_\kappa$ as a subset and collapse it. Since $V_\kappa$
satisfies that $\her{\alpha}$ exists for every $\alpha<\kappa$, it
follows by elementarity that $\her{\kappa}$ exists in $W$.

Let $M=\her{\kappa}^W$ and observe that $M$ is a weak $\kappa$-model
containing $A$. Define $U=\{B\in \x\mid \kappa\in B^*\}$. We claim
that $U$ is a \hbox{1-good} $M$-ultrafilter. By Proposition
\ref{prop:ultdef}, we will be done if we can verify the claim. Thus,
we need to verify that $U$ is a weakly amenable $M$-ultrafilter and
that the ultrapower by $U$ is well-founded.

It is clear that $\la M,\in, U\ra\models ``U$ is an ultrafilter". To
check that $\la M,\in, U\ra\models ``U$ is normal", fix a regressive
$f:B\to \kappa$ in $M$ for some $B\in U$. Since we can code $f$ as a
subset of $\kappa$ and $\Power(\kappa)^W=\x$, we can think of $f$ as
being in $\x$. Now we can consider the regressive $f^*:B^*\to
\kappa^*$ and let $f^*(\kappa)=\alpha<\kappa$ since $\kappa\in B^*$.
Thus, $\kappa\in C^*$ where $C=\{\xi\in\kappa\mid f(\xi)=\alpha\}$,
and hence $C\in U$. Thus $\la M,\in, U\ra\models ``U$ is normal". It
is a standard exercise to show that a normal ultrafilter on $\kappa$
containing all the tail subsets of $\kappa$ is $\kappa$-complete,
and therefore since this property easily holds of $U$, we have $\la
M,\in,U\ra\models``U$ is $\kappa$-complete". This completes the
argument that $U$ is an $M$-ultrafilter.

To show that $U$ is weakly amenable, fix $\la
B_\alpha\mid\alpha<\kappa\ra$ a sequence in $M$ of subsets of
$\kappa$. We need to see that the set $C=\{\alpha\in\kappa\mid
B_{\alpha}\in U\}$ is in $M$. Again, since we can code the sequence
$\la B_\alpha\mid \alpha<\kappa\ra$ as a subset of $\kappa$, we
think of it as being in $\x$. Thus, in $W$, we can define the set
$\{\alpha\in\kappa\mid \kappa\in B^*_\alpha\}$, and it is clear that
this set is exactly $C$.

It remains to show that the ultrapower of $M$ by $U$ is
well-founded. It will help first to verify that if $C\in\x$ codes a
well-founded relation on $\kappa$, then $C^*$ codes a well-founded
relation on $\text{Ord}^W$. If $C\in\x$ codes a well-founded
relation, \hbox{$\la V_\kappa,\in,B\ra_{B\in\x}$} satisfies that
$C\upharpoonright\alpha$ has a rank function for all
$\alpha<\kappa$. It follows that \hbox{$\la W,\in, B^*\ra_{B\in\x}$}
satisfies that $C^*\upharpoonright \alpha$ has a rank function for
all $\alpha<\text{Ord}^W$. Since $\kappa$ is weakly compact and we
assumed that $W$ has size $\kappa$, we can find a
\emph{well-founded} elementary extension $\la X, E,
B^{**}\ra_{B\in\x}$ for the structure $\la W,\in,B^*\ra_{B\in\x}$
satisfying that there exists an ordinal above the ordinals of $W$
(Theorem \ref{th:wc}(1)). There is no reason to expect that $X$ is
an end-extension or that $E$ is the true membership relation, but
that is not important for us. We only care that $E$ is well-founded
and $X$ thinks it has an ordinal $>\text{Ord}^W$. By elementarity,
it follows that $\la X,E, B^{**}\ra_{B\in\x}$ satisfies that
$C^{**}\upharpoonright \alpha$ has a rank function for all
$\alpha<\text{Ord}^X$. In particular, if $\alpha>\text{Ord}^W$ in
$X$, then $\la X,E, B^{**}\ra_{B\in\x}$ satisfies that
$C^{**}\upharpoonright \alpha$ has a rank function. Since the
structure $\la X,E, B^{**}\ra_{B\in\x}$ is well-founded and can only
add new elements to $C^*$, if $C^*$ was not well-founded to begin
with, $X$ would detect this. Hence $C^*$ is well-founded.

Now we go back to proving that the ultrapower of $M$ by $U$ is
well-founded. Suppose towards a contradiction that there exists a
membership descending sequence $\ldots E\,[f_n]\,E\ldots
E\,[f_1]\,E\,[f_0]$ of elements of the ultrapower. Each
$f_n:\kappa\to M$ is an element of $M$, and for every $n\in\omega$,
the set $A_n=\{\alpha\in\kappa\mid f_{n+1}(\alpha)\in
f_n(\alpha)\}\in U$. In $M$, we will define subsets $E_n$ and $F_n$
of $\kappa$ coding information about $f_n$. Fix a transitive set
$T_n$ in $M$ such that the range of $f_n$ is contained in $T_n$ and
let $E_n$ be a relation on $\kappa$ coding $T_n$. Let $F_n$ be a
function on $\kappa$ such that $F_n(\alpha)$ is the element
representing $f_n(\alpha)$ in $E_n$. Next, we define subsets $B_n$
of $\kappa$ as follows. If $\alpha\in A_n$, then $f_{n+1}(\alpha)\in
f_n(\alpha)$ and therefore the transitive closure of
$f_{n+1}(\alpha)$ is a subset of the transitive closure of
$f_n(\alpha)$. Thus, for $\alpha\in A_n$, there is a membership
preserving map $\varphi_\alpha^n$ mapping a transitive subset of
$E_{n+1}$ containing $F_{n+1}(\alpha)$ onto a transitive subset of
$E_n$ such that $\varphi_\alpha^n(F_{n+1}(\alpha))$ is an $E_n$
element of $F_n(\alpha)$. Let $B_n$ code a collection of such maps
$\varphi_\alpha^n$ for $\alpha\in A_n$.

By the observation above, each $E_n^*$ codes a well-founded relation
on $\kappa^*$. Since each $A_n\in U$, it follows that $\kappa\in
A_n^*$ and so, by elementarity, $B_n^*$ codes a membership
preserving map $\varphi^n_\kappa$ from a transitive subset of
$E_{n+1}^*$ to a transitive subset of $E_n^*$ such that
$\varphi^n_\kappa(F^*_{n+1}(\kappa))$ is an $E_n^*$ element of
$F^*_n(\kappa)$. If we let $\pi_n$ be the Mostowski collapse of
$E_n^*$, then by the uniqueness of the Mostowski collapse, we have
that
$\pi_{n+1}(F^*_{n+1}(\kappa))=\pi_n(\varphi^n_\kappa(F^*_{n+1}(\kappa)))$
and by the definition of $\varphi^n_\kappa$, we have that
$\pi_n(\varphi^n_\kappa(F^*_{n+1}(\kappa)))$ is an element of
$\pi_n(F^*_n(\kappa))$. It follows that the elements
$\pi_n(F^*_n(\kappa))$ form a descending $\in$-sequence. Thus, we
have reached a contradiction showing that the ultrapower of $M$ by
$U$ is well-founded.
\end{proof}
\section{Ramsey Cardinals}\label{sec:ramsey}
Recall that Theorem \ref{th:ramseycard} gave a characterization of
Ramsey cardinals in terms of the existence of elementary embeddings
for weak $\kappa$-models of set theory. For convenience, we restate
the theorem below.
\begin{theo}
A cardinal $\kappa$ is Ramsey if and only if every
$A\subseteq\kappa$ is contained in a weak $\kappa$-model $M$ for
which there exists a $\kappa$-powerset preserving elementary
embedding $j:M\to N$ with the additional property that whenever $\la
A_n\mid n\in \omega\ra$ are subsets of $\kappa$ such that each
$A_n\in M$ and $\kappa\in j(A_n)$, then $\cap_{n\in\omega} A_n \neq
\emptyset$.
\end{theo}
Proposition \ref{prop:ultdef}(3) restated the characterization in
terms of the existence of $M$-ultrafilters. Again, we restate it
here for convenience.
\begin{pro}
A cardinal $\kappa$ is Ramsey if and only if every
$A\subseteq\kappa$ is contained in a weak $\kappa$-model $M$ for
which there exists a weakly amenable countably complete
$M$-ultrafilter on $\kappa$.
\end{pro}
In Theorem \ref{th:backwardram}, we proved the backward direction of
Proposition \ref{prop:ultdef}(3). In this section, we give a much
more involved proof of the forward direction, essentially following
\cite{dodd:coremodel}. We conclude the section with another
embedding characterization for Ramsey and strongly Ramsey cardinals.
\begin{definition}\label{d:good}
Suppose $\kappa$ is a cardinal and $A\subseteq\kappa$. Then
$I\subseteq\kappa$ is a \emph{good set of indiscernibles} for
$\lset{\kappa}$ if for all $\gamma\in I$:
\begin{itemize}
\item[(1)] $\lsetr{\gamma}\prec \lset{\kappa}$.
\item[(2)] $I\setminus\gamma$ is a set of indiscernibles for
$\lsetex{\kappa}_{\xi\in\gamma}$.
\end{itemize}
\end{definition}
It turns out that if $I$ is a good set of indiscernibles for
$\lset{\kappa}$, then every $\gamma\in I$ is inaccessible in
$L_\kappa[A]$. Thus, in particular, $L_\kappa[A]$ is the union of an
elementary chain of models of ZFC, and hence is itself a model of
ZFC.

\begin{lemma}\label{le:indiscernibles}
If $\kappa$ is Ramsey and $A\subseteq\kappa$, then $\lset{\kappa}$
has a good set of indiscernibles of size $\kappa$.
\end{lemma}
For a proof, see \cite{dodd:coremodel} (Ch. 17). We are now ready to
prove that if $\kappa$ is Ramsey, then every $A\subseteq\kappa$ can
be put into a weak $\kappa$-model $M$ for which there exists a
weakly amenable countably complete $M$-ultrafilter on $\kappa$.
\begin{proof}[Proof of the forward direction of Proposition \ref{prop:ultdef}(3)]
Fix $A\subseteq\kappa$ and consider the structure
$\sa=\lset{\kappa}$. Note that $\sa$ has definable Skolem functions
since it has a definable well-ordering. By Lemma
\ref{le:indiscernibles}, $\sa$ has a good set indiscernibles $I$ of
size $\kappa$. For every $\gamma\in I$ and $n\in\omega$, let
$\arr{\gamma}_n$ denote the sequence
$\gamma_1<\gamma_2<\cdots<\gamma_n$ of the first $n$ indiscernibles
in $I$ above $\gamma$. Given $\gamma\in I$ and $n\in\omega$, let
$\ahat{\gamma}{n}=\la\mhat{\gamma}{n},A\cap
\mhat{\gamma}{n}\ra=Scl_{\sa}(\gamma+1\cup \arr{\gamma}_n)$ be the
Skolem closure using the definable Skolem functions of $\sa$. Since
$L_\kappa[A]$ is a model of ZFC, it satisfies that for every
$\lambda$, $H_\lambda$ exists and is a model of $\rm{ZFC}^-$. Since
$\ahat{\gamma}{n}\prec\sa$ and $\gamma\in \mhat{\gamma}{n}$, we have
$\her{\gamma}^{\sa}\in \mhat{\gamma}{n}$. Next, let
$\mh{\gamma}{n}=\mhat{\gamma}{n}\cap \her{\gamma}^{\sa}$ and
$\ah{\gamma}{n}=\la \mh{\gamma}{n},A\cap \mh{\gamma}{n}\ra$.
\begin{sublemma}
Each $\mh{\gamma}{n}$ is a transitive model of ${\rm ZFC}^-$.
\end{sublemma}
\begin{proof}
It is clear that $\mh{\gamma}{n}\models{\rm ZFC}^-$ since it is
precisely the $\her{\gamma}$ of $\ahat{\gamma}{n}$ and
$\ahat{\gamma}{n}$ knows by elementarity that $\her{\gamma}$
satisfies $\rm{ZFC}^-$.

To see that $\mh{\gamma}{n}$ is transitive, fix $a\in\mh{\gamma}{n}$
and $b\in a$. The set $a$ is an element of $\her{\gamma}^{\sa}$ and
is therefore coded by a subset of $\gamma\times\gamma$ in $\sa$. By
elementarity, $\mhat{\gamma}{n}$ contains a set
$E\subseteq\gamma\times\gamma$ coding $a$ and the Mostowski collapse
$\pi:\la \gamma,E\ra\to Trcl(a)$. Let $\alpha\in \gamma$ such that
$\sa\models \pi(\alpha)=b$. Since $\alpha<\gamma$, we have
$\alpha\in \mhat{\gamma}{n}$, and so $b\in \mhat{\gamma}{n}$ by
elementarity. Since clearly $b\in \her{\gamma}^{\sa}$, we have
$b\in\mh{\gamma}{n}$.
\end{proof}
\begin{sublemma}
For every $\gamma\in I$ and $n\in\omega$, we have
$\ah{\gamma}{n}\prec\ah{\gamma}{n+1}$.
\end{sublemma}
\begin{proof}
It suffices to observe that
$\ahat{\gamma}{n}\prec\ahat{\gamma}{n+1}$ together with the fact
that $\ah{\gamma}{n}=\la\her{\gamma}^{\ahat{\gamma}{n}},A\cap
\her{\gamma}^{\ahat{\gamma}{n}}\ra$ and
$\ah{\gamma}{n+1}=\la\her{\gamma}^{\ahat{\gamma}{n+1}},A\cap\her{\gamma}^{\ahat{\gamma}{n+1}}\ra$.
\end{proof}
Recall that if $a\in\mhat{\gamma}{n}$, then
$a=h(\xi_0,\ldots,\xi_m,\gamma,\arr{\gamma}_n)$ where $h$ is a
definable Skolem function, $\xi_i\in\gamma$, and $\arr{\gamma}_n$
are the first $n$ indiscernibles above $\gamma$ in $I$. Given
$\gamma<\delta\in I$, define
$\map{\gamma}{\delta}{n}:\mhat{\gamma}{n}\to \mhat{\delta}{n}$ by
$\map{\gamma}{\delta}{n}(a)=h(\xi_0,\ldots,\xi_m,\delta,\arr{\delta}
_n)$ where $a=h(\xi_0,\ldots,\xi_m,\gamma,\arr{\gamma}_n)$ is as
above. Observe that since $I\setminus\gamma$ are indiscernibles for
$\lsetex{\kappa}_{\xi\in\gamma}$ (Definition \ref{d:good}(2)), the
map $\map{\gamma}{\delta}{n}$ is clearly well-defined. It is,
moreover, an elementary embedding of the structure
$\ahat{\gamma}{n}$ into $\ahat{\delta}{n}$. Since
$\map{\gamma}{\delta}{n}(\gamma)=\delta$ and
$\map{\gamma}{\delta}{n}(\xi)=\xi$ for all $\xi<\gamma$, the
critical point of $\map{\gamma}{\delta}{n}$ is $\gamma$. Finally,
note that for all $\gamma<\delta<\beta\in I$, we have
$\map{\gamma}{\beta}{n}\circ\map{\beta}{\delta}{n}=
\map{\gamma}{\delta}{n}$.
\begin{sublemma}
The map
$\map{\gamma}{\delta}{n}\upharpoonright\mh{\gamma}{n}:\mh{\gamma}{n}
\to\mh{\delta}{n}$ is an elementary embedding of the structure
$\ah{\gamma}{n}$ into $\ah{\delta}{n}$.
\end{sublemma}
\begin{proof}
Fix $a\in \mh{\gamma}{n}$ and recall that $\mhat{\gamma}{n}$ thinks
$a\in\her{\gamma}$. By elementarity of the
$\map{\gamma}{\delta}{n}$, it follows that $\mhat{\delta}{n}$ thinks
$\map{\gamma}{\delta}{n}(a)\in \her{\delta}$. Therefore
$\map{\gamma}{\delta}{n}:\mh{\gamma}{n}\to\mh{\delta}{n}$.
Elementarity follows as in the previous lemma.
\end{proof}
For $\gamma\in I$, define $\ult{\gamma}{n}=\{X\in \Power(\gamma)^{
\ah{\gamma}{n}}\mid \gamma\in\map{\gamma}{\delta}{n}(X)\text{ for
some }\delta>\gamma\}$. Equivalently, we could have used ``for
\emph{all} $\delta>\gamma$" in the definition.
\begin{sublemma}
$\ult{\gamma}{n}$ is an $\mh{\gamma}{n}$-ultrafilter on $\gamma$.
\end{sublemma}
\begin{proof}
Easy.
\end{proof}
Observe that if $a_0,\ldots,a_n\in \lone{\gamma}$ for some
$\gamma\in I$, then for every formula $\varphi(\arr{x})$, we have
$\lset{\kappa}\models\varphi(\arr{a})\leftrightarrow
\lsetr{\gamma}\models\varphi(\arr{a})\leftrightarrow\lset{\kappa}
\models``\lsetr{\gamma}\models \varphi(\arr{a})"$. It follows that
for every $\gamma\in I$, the model $\lset{\kappa}$ has a truth
predicate for formulas with parameters from $\lone{\gamma}$ that is
definable with $\gamma$ as a parameter.
\begin{sublemma}
The $\mh{\gamma}{n}$-ultrafilter $\ult{\gamma}{n}$ is an element of
$\mh{\gamma}{n+2}$.
\end{sublemma}
\begin{proof}
In $\ahat{\gamma}{n+2}$, we have $\ult{\gamma}{n}=\{x\in
\Power(\gamma)\mid \exists h\,\exists\xi_0,\ldots,\xi_m<\gamma\,\,h$
is a Skolem term and
$h=(\xi_0,\ldots,\xi_m,\gamma,\gamma_1,\ldots,\gamma_n)\wedge
 \gamma\in
h(\xi_0,\ldots,\xi_m,\gamma_1,\gamma_2,\ldots,\gamma_{n+1})\}$. This
follows since $\gamma_{n+2}\in \mhat{\gamma}{n+2}$, and therefore we
can define a truth predicate for $\lone{\gamma_{n+2}}$, which is
good enough for the definition above. So far we have shown that
$\ult{\gamma}{n}$ is in $\mhat{\gamma}{n+2}$, but obviously
$\ult{\gamma}{n}\in \her{\gamma}^{\sa}$, and therefore
$\ult{\gamma}{n}\in \mh{\gamma}{n+2}$.
\end{proof}
It should be clear that $\ult{\gamma}{n}\subseteq\ult{\gamma}{n+1}$
and $\map{\gamma}{\delta}{n}\subseteq\map{\gamma}{\delta}{n+1}$. Let
$\mhh{\gamma}=\cup_{n\in\omega}\mh{\gamma}{n}$ and $\ahh{\gamma}=\la
\mhh{\gamma},A\cap \mhh{\gamma}\ra$. Note that $\mhh{\gamma}$ is a
transitive model of $\rm{ZFC}^-$. Let
$U_\gamma=\cup_{n\in\omega}\ult{\gamma}{n}$ and
$\ma{\gamma}{\delta}=\cup_{n\in\omega}\map{\gamma}{\delta}{n}:
\mhh{\gamma}\to\mhh{\delta}$. The map $\ma{\gamma}{\delta}$ is an
elementary embedding from the structure $\ahh{\gamma}$ into
$\ahh{\delta}$ and $U_\gamma$ is an $\mhh{\gamma}$-ultrafilter on
$\gamma$.
\begin{sublemma}
The $\mhh{\gamma}$-ultrafilter $U_\gamma$ is weakly amenable.
\end{sublemma}
\begin{proof}
Fix a sequence $\la B_\alpha\mid \alpha\in\gamma\ra$ in
$\mhh{\gamma}$ of subsets of $\gamma$. We need to show that
$C=\{\xi\in\gamma\mid B_\xi\in U_\gamma\}$ is an element of
$\mhh{\gamma}$. Since $\la
B_\alpha\mid\alpha<\kappa\ra\in\mhh{\gamma}$, it follows that $\la
B_\alpha\mid\alpha<\kappa\ra\in\mh{\gamma}{n}$ for some
$n\in\omega$. But then $C=\{\xi\in\gamma\mid
B_\xi\in\ult{\gamma}{n}\}$ and $\ult{\gamma}{n}\in
\mh{\gamma}{n+2}\subseteq\mhh{\gamma}$.
\end{proof}
Now for every $\gamma\in I$, we have an associated structure
$\la\mhh{\gamma},\in,A\cap \mhh{\gamma},U_\gamma\ra$. Also, if
$\gamma<\delta$ in $I$, we have an elementary embedding
$\ma{\gamma}{\delta}:\mhh{\gamma}\to\mhh{\delta}$ with critical
point $\gamma$ between the structures $ \ahh{\gamma}$ and $
\ahh{\delta}$ such that $X\in U_\gamma$ if and only if
$\ma{\gamma}{\delta}(X)\in U_\delta$. This is a directed system of
embeddings, and so we can take its direct limit. Define $\la B,E,
A', V\ra=\lim_{\gamma\in I}\la \mhh{\gamma},\in, A\cap\mhh{\gamma},
U_\gamma\ra$.
\begin{sublemma}\label{le:well-founded}
The relation $E$ on $B$ is well-founded.
\end{sublemma}
\begin{proof}
The elements of $B$ are functions $t$ with domains $\{\xi\in I\mid
\xi\geq\alpha\}$ for some $\alpha\in I$ satisfying the properties:
\begin{itemize}
\item[(1)] $t(\gamma)\in \mhh{\gamma}$,
\item[(2)] for $\gamma<\delta$ in domain of $t$, we have
$t(\delta)=\ma{\gamma}{\delta}(t(\gamma))$,
\item[(3)] there is no $\xi\in I\cap\alpha$ for which there is
$a\in \mhh{\xi}$ such that  $\ma{\xi}{\alpha}(a)=t(a)$.
\end{itemize}
Note that each $t$ is determined once you know any $t(\xi)$ by
extending uniquely forward and backward. Standard arguments (for
example, \cite{jech:settheory}, Ch. 12) show that $\la
B,E,A'\ra\models \varphi(t_1,\ldots,t_n)\leftrightarrow\exists
\gamma\,\ahh{\gamma}\models \varphi(t_1(\gamma),\ldots,t_n(\gamma))
\leftrightarrow$ for all $\gamma$ in the intersection of the domains
of the $t_i$, the structure
$\ahh{\gamma}\models\varphi(t_1(\gamma),\ldots,t_n(\gamma))$. Note
that this truth definition holds only of atomic formulas where the
formulas involve the predicate for the ultrafilter.

Suppose to the contrary that $E$ is not well-founded, then there is
a descending $E$-sequence $\cdots E\, t_n\, E\cdots E \,t_1 \,E\,
t_0$. Find $\gamma_0$ such that $\ahh{\gamma_0}\models
t_1(\gamma_0)\in t_0(\gamma_0)$. Next, find $\gamma_1>\gamma_0$ such
that $\ahh{\gamma_1}\models t_2(\gamma_1)\in t_1(\gamma_1)$. In this
fashion, define an increasing sequence
$\gamma_0<\gamma_1<\cdots<\gamma_n<\cdots$ such that
$\ahh{\gamma_n}\models t_{n+1}(\gamma_n)\in t_n(\gamma_n)$. Let
$\gamma\in I$ such that $\gamma>\text{sup}_{n\in\omega}\gamma_n$. It
follows that for all $n\in\omega$, the structure\break
$\ahh{\gamma}\models
\ma{\gamma_n}{\gamma}(t_{n+1}(\gamma_n))\in\ma{\gamma_n}{\gamma}
(t_n(\gamma_n))$, and therefore $\ahh{\gamma}\models
t_{n+1}(\gamma)\in t_n(\gamma)$. But, of course, this is impossible.
Thus, $E$ is well-founded.
\end{proof}
Let $\la M,\in, A^*, U\ra$ be the Mostowski collapse of $\la B,E,
A', V\ra$.
\begin{sublemma}
The cardinal $\kappa$ is an element of  $M$.
\end{sublemma}
\begin{proof}
Fix $\alpha\in\kappa$ and let $\gamma$ be the least ordinal in $I$
above $\alpha$. Let $t_\alpha$ have domain $\{\xi\in I\mid
\xi\geq\gamma\}$ with $t_\alpha(\xi)=\alpha$. The function
$t_\alpha$ is an element of $B$ which collapses to $\alpha$. Let
$t_\kappa$ have domain $I$ with $t_\kappa(\gamma)=\gamma$. The
function $t_\kappa$ is an element of $B$ which collapses to
$\kappa$.
\end{proof}
Let $j_\gamma:\mhh{\gamma}\to M$ such that $j_\gamma(a)$ is the
collapse of the function $t$ for which $t(\gamma)=a$. The map
$j_\gamma$ is an elementary embedding of the structure
$\ahh{\gamma}$ into  $\la M,\in,A^*\ra$. It is, moreover, elementary
for atomic formulas in the language with the predicate for the
ultrafilter. Observe that $j_\gamma(\xi)=\xi$ for all $\xi<\gamma$
since if $t(\gamma)=\xi$, then $t=t_\xi$. Also,
$j_\gamma(\gamma)=\kappa$ since if $t(\gamma)=\gamma$, then
$t=t_\kappa$. So the critical point of each $j_\gamma$ is $\kappa$.
Finally, if $\gamma<\delta$ in $I$, then $j_\delta\circ
\ma{\gamma}{\delta}=j_\gamma$.
\begin{sublemma}
The set $U$ is a weakly amenable $M$-ultrafilter on $\kappa$.
\end{sublemma}
\begin{proof}
Easy.
\end{proof}
\begin{sublemma}
A set $X\in U$ if and only if there exists $\alpha\in I$ such
that\break $\{\xi\in I\mid \xi>\alpha\}\subseteq X$.
\end{sublemma}
\begin{proof}
Fix $X\subseteq\kappa$ in $M$ and $\beta\in I$ such that for all
$\xi>\beta$, there is $X'\in \mhh{\xi}$ with $j_\xi(X')=X$. For
$\xi>\beta$, we have $X\in U\leftrightarrow X'\in U_\xi
\leftrightarrow \xi\in\ma{\xi}{\xi_1}(X')\leftrightarrow
j_{\xi_1}(\xi)\in
j_{\xi_1}\circ\ma{\xi}{\xi_1}(X')=j_\xi(X')\leftrightarrow \xi\in j_
\xi(X')=X$. Thus, for $\alpha>\beta$, we have $\{\xi\in I\mid
\xi>\alpha\}\subseteq X$ if and only if $X\in U$.
\end{proof}
\begin{sublemma}
The $M$-ultrafilter $U$ is countably complete.
\end{sublemma}
\begin{proof}
Fix $\la A_n\mid n\in\omega\ra$ a sequence of elements of $U$. We
need to show that $\cap_{n\in\omega}A_n\neq\emptyset$. For each
$A_n$, there exists $\gamma_n\in I$ such that $X_n=\{\xi\in I\mid
\xi>\gamma_n\}\subseteq A_n$. Thus,
$\cap_{n\in\omega}X_n\subseteq\cap_{n\in\omega}A_n$ and clearly
$\cap_{n\in\omega}X_n$ has size $\kappa$.
\end{proof}
It remains to show that $A\in M$.
\begin{sublemma}
The set $A^*\upharpoonright \kappa=A$, and hence $A\in M$.
\end{sublemma}
\begin{proof}
Fix $\alpha\in A$ and let $\gamma\in I$ such that $\gamma>\alpha$,
then $\ahh{\gamma}\models \alpha\in A$. It follows that $\la
M,\in,A^*\ra\models j_\gamma(\alpha)\in A^*$, but
$j_\gamma(\alpha)=\alpha$, and so $\alpha\in A^*$. Thus, $A\subseteq
A^*$. Now fix $\alpha\in A^*\upharpoonright\kappa$ and let
$\gamma\in I$ such that $\gamma>\alpha$, then
$j_\gamma(\alpha)=\alpha$, and so $j_\gamma(\alpha)\in A^*$. It
follows that $\alpha\in A$. Thus, $A^*\upharpoonright\kappa\subseteq
A$. We conclude that $A=A^*\upharpoonright\kappa$.
\end{proof}
To summarize, we have shown that $M$ is a weak $\kappa$-model since
it is a transitive model of $\rm{ZFC}^-$ of size $\kappa$ containing
$\kappa$ as an element. We have further shown that $A$ is an element
of $M$ and $U$ is a countably complete weakly amenable
$M$-ultrafilter.
\end{proof}
We conclude this section with some basic facts about finite products
and iterations of weakly amenable countably complete
$M$-ultrafilters.
\begin{lemma}\label{le:iterate}
Suppose $M$ is a weak $\kappa$-model, $U$ is a $1$-good
$M$-ultrafilter on $\kappa$, and $j:M\to N$ is the ultrapower by
$U$. Then  $j_{U}(U)=\{A\subseteq j_{U}(\kappa)\mid A=[f]_{U}\text{
and }\{\alpha\in\kappa\mid f(\alpha)\in U\}\in U\}$ is a weakly
amenable $N$-ultrafilter on $j_{U}(\kappa)$ containing $j''U$ as a
subset.
\end{lemma}
For proof, see \cite{kanamori:higher} (Ch. 4, Sec. 19). Lemma
\ref{le:iterate} is essentially saying that we can take the
ultrapower of the structure $\la M,\in, U\ra$ by the $M$-ultrafilter
$U$ and the \L o\'{s} Theorem still goes through due to the weak
amenability of $U$. Previously, we only took ultrapowers of the
structures $\la M,\in\ra$. The \L o\'{s} Theorem there relied on the
fact that $M\models {\rm ZFC^-}$, but the structure $\la M,\in,U\ra$
need not satisfy any substantial fragment of ZFC. Thus, the
additional assumption of weak amenablity is precisely what is
required to carry out the argument. The ultrafilter $j_{U}(U)$ is
simply the relation corresponding to $U$ in the ultrapower.

The next lemma is an adaptation to the case of weak $\kappa$-models
of the standard fact from iterating ultrapowers.
\begin{lemma}\label{le:diagram}
Suppose $M$ is a weak $\kappa$-model, $U$ is a $1$-good
$M$-ultrafilter on $\kappa$, and $j_U:M\to M/U$ is the ultrapower by
$U$. Suppose further that
\begin{displaymath}
j_{U^n}:M\to M/U^n\text{ \emph{and} }h_{U^n}:M/U\to (M/U)/U^n
\end{displaymath}
are the well-founded ultrapowers by $U^n$. Then the ultrapower
\begin{displaymath}
j_{j_{U^n}(U)}:M/U^n\to (M/U^n)/j_{U^n}(U)\text{ \emph{(by}
}j_{U^n}(U))
\end{displaymath}
and the ultrapower
\begin{displaymath}
j_{U^{n+1}}:M\to M/U^{n+1}\text{ \emph{(by} }U^{n+1})
\end{displaymath}
are also well-founded.
Moreover,
\begin{displaymath}
 (M/U^n)/j_{U^n}(U)=(M/U)/U^n=M/U^{n+1},
\end{displaymath}
  and the
following diagram commutes:
\begin{diagram}
M & &\rTo^{j_U}& &M/U\\
 & \rdTo(4,4)^{j_{U^{n+1}}}& & &\\
\dTo^{j_{U^n}}& & & &\dTo_{h_{U^n}}\\
 & & & & \\
M/U^n & &\rTo_{j_{j_{U^n}(U)}}& &(M/U^n)/j_{U^n}(U)
\end{diagram}
\end{lemma}
\begin{proof}
The idea is to define obvious isomorphisms between
$(M/U^n)/j_{U^n}(U)$ and $M/U^{n+1}$ and between $(M/U)/U^n$ and
$M/U^{n+1}$.
\end{proof}
\begin{proposition}
If $M$ is a weak $\kappa$-model and $U$ is a weakly amenable
countably complete $M$-ultrafilter on $\kappa$, then the ultrapowers
of $M$ by $U^n$ are well-founded for all $n\in\omega$.
\end{proposition}
\begin{proof}
Use Lemma \ref{le:diagram} and argue by induction on $n$.
\end{proof}
The commutative diagram above gives an interesting reformulation of
the embeddings for Ramsey and strongly Ramsey cardinals.
\begin{proposition}
A cardinal $\kappa$ is Ramsey if and only if every
$A\subseteq\kappa$ is contained in a weak $\kappa$-model $M\models
{\rm ZFC}$ for which there exists a weakly amenable countably
complete $M$-ultrafilter with the ultrapower $j:M\to N$ having
$M\prec N$.
\end{proposition}
The difference from the earlier embeddings is that now for every
$A\subseteq\kappa$, we have $A\in M$ where $M$ is a model of full
{\rm ZFC} and $j:M\to N$ is an embedding such that not only do $M$
and $N$ have the same subsets of $\kappa$, but actually $M\prec N$.
\begin{proof}
Fix $A\subseteq\kappa$ and choose a weak $\kappa$-model $M$
containing $A$ and $V_\kappa$ for which there exists a weakly
amenable countably complete $M$-ultrafilter $U$ on $\kappa$.
Let\break $j:M\to N$ be the ultrapower by $U$. The commutative
diagram from Lemma \ref{le:diagram} becomes the following for the
case $n=1$:
\begin{diagram}
M & &\rTo^j& &N=M/U\\
 & \rdTo(4,4)^{j_{U^2}}& & &\\
\dTo^j& & & &\dTo_{h_U}\\
 & & & & \\
N=M/U & &\rTo_{j_{j(U)}}& & K=N/j(U)
\end{diagram}
Let $M'=V_{j(\kappa)}^N$ and observe that it is a transitive model
of {\rm ZFC}. Let
$K'=V_{j_{j(U)}(j(\kappa))}^K=V_{h_U(j(\kappa))}^K$.  Since
$j(\kappa)$ is regular in $N$, the models $M'$ and $N$ have the same
functions from $\kappa$ to $M'$, and thus the map
$h_U\upharpoonright M':M'\to K'$ is the ultrapower embedding of $M'$
into $M'/U$. It remains to show that \hbox{$M'\prec K'$}, but this
follows easily from the $j_{j(U)}$ side of the commutative diagram
since \hbox{$j_{j(U)}\upharpoonright M':M'\to K'$} and $j_{j(U)}$ is
identity on $M'$.
\end{proof}
\begin{corollary}
A cardinal $\kappa$ is strongly Ramsey if and only if every
$A\subseteq\kappa$ is contained in a $\kappa$-model $M\models{\rm
ZFC}$ for which there exists an elementary embedding $j:M\to N$ with
critical point $\kappa$ such that $M\prec N$.
\end{corollary}
\begin{proof}
Using the previous proof it suffices to verify that $M'$ is a
$\kappa$-model, that is, it is closed under $<\kappa$-sequences. By
Proposition \ref{prop:wlog}, we can assume, without loss of
generality, that $N$ is closed under $<\kappa$-sequences. Since $N$
thinks that $j(\kappa)$ is inaccessible, it follows that
$M'=V_{j(\kappa)}^N$ must be closed under $<\kappa$-sequences as
well.
\end{proof}
The same arguments cannot be carried out for weakly Ramsey cardinals
since the ultrapower by $U^2$ need not be well-founded in that case.
Thus, it is not clear whether the weakly Ramsey cardinals have a
similar characterization.
\begin{question}
If $\kappa$ is a weakly Ramsey cardinal, does there exist for every
$A\subseteq\kappa$, a weak $\kappa$-model $M\models{\rm ZFC}$ and an
elementary embedding  $j:M\to N$ with critical point $\kappa$ such
that $M\prec N$?
\end{question}
\section{$\alpha$-iterable Cardinals and $\alpha$-good
Ultrafilters}\label{sec:iterable} A key feature of $M$-ultrafilters
associated with Ramsey-like embeddings is weak amenability; meaning
that they have the potential to be iterated. Recall that
\hbox{0-good} $M$-ultrafilters are exactly the ones with
well-founded ultrapowers, and \hbox{1-good} $M$-ultrafilters are
weakly amenable and 0-good. Consider starting with a weak
$\kappa$-model $M_0$ for which there exists a \hbox{1-good}
$M_0$-ultrafilter $U_0$. Applying Lemma \ref{le:iterate}, take the
ultrapower of the structure $\la M_0,\in, U_0\ra$ by $U_0$ to obtain
the structure $\la M_1, \in,U_1\ra$ with $U_1$ a weakly amenable
$M_1$-ultrafilter. Next, if the result happens to be
\emph{well-founded}, take the ultrapower of $\la M_1, \in, U_1\ra$
by $U_1$ to obtain the structure $\la M_2,\in, U_2\ra$. We will call
such $U_0$ $2$-\emph{good} to indicate that we were able to do a
two-step iteration. If $\xi\leq\omega$ and we can continue iterating
for $\xi$-many steps by obtaining well-founded ultrapowers, we will
say that the $M$-ultrafilter $U_0$ is $\xi$-\emph{good}. Suppose
that $U_0$ is $\omega$-good. In this case, we have a directed system
of models $\la M_n,\in, U_n\ra$ with the corresponding ultrapower
embeddings. Take the direct limit of this system. If the direct
limit happens to be well-founded, collapse it to obtain $\la
M_\omega,\in, U_\omega\ra$, where $U_\omega$ is a weakly amenable
$M_\omega$-ultrafilter. We will call $U_0$ $\omega+1$-\emph{good} to
indicate that we were able to do an $\omega+1$-step iteration. We
can proceed in this fashion as long as the iterates are
well-founded.\footnote{See \cite{kanamori:higher} (Ch. 4, Sec. 19)
for details involved in this construction.}
\begin{definition}
Suppose $M$ is a weak $\kappa$-model. An $M$-ultrafilter on $\kappa$
is $\alpha$-\emph{good}, if the ultrapower construction can be
iterated $\alpha$-many steps.
\end{definition}
\begin{definition}
A cardinal $\kappa$ is $\alpha$-\emph{iterable} if every
$A\subseteq\kappa$ is contained in a weak $\kappa$-model $M$ for
which there exists an $\alpha$-good $M$-ultrafilter on $\kappa$.
\end{definition}
Weakly Ramsey cardinals are exactly the 1-iterable cardinals.
Gaifman showed in \cite{gaifman:ultrapowers} that if an
$M$-ultrafilter is $\omega_1$-good, then it is already $\alpha$-good
for every ordinal $\alpha$. Kunen showed in \cite{kunen:ultrapowers}
that if an $M$-ultrafilter is weakly amenable and countably
complete, then it is $\omega_1$-good. Thus, Ramsey cardinals are
$\omega_1$-iterable. Even though countable completeness is
sufficient for full iterability, it is not necessary. In fact,
Sharpe and Welch showed in \cite{welch:ramsey} that
$\omega_1$-iterable cardinals are strictly weaker than Ramsey
cardinals.
\begin{theorem}
An $\omega_1$-Erd\H{o}s is a limit of $\omega_1$-iterable cardinals.
\end{theorem}
In an upcoming paper with Welch \cite{gitman:welch}, we show:
\begin{theorem}\label{th:alphaiterable}
The $\alpha$-iterable cardinals forms a strict hierarchy for
$\alpha\leq\omega_1$. In particular, for $\alpha<\beta\leq\omega_1$,
a $\beta$-iterable cardinal is a limit of $\alpha$-iterable
cardinals.
\end{theorem}
\begin{theorem}
The $\alpha$-iterable cardinals are downwards absolute to $L$ for
$\alpha<\omega_1^L$.
\end{theorem}
\bibliographystyle{alpha}
\bibliography{gitmanbib,logicbib}

\begin{thebibliography}{Cum10}

\bibitem[Cum10]{cummings:weaklycompact}
J.~Cummings.
\newblock Iterated forcing and elementary embeddings.
\newblock In M.~Foreman and A.~Kanamori, editors, {\em Handbook of Set Theory},
  pages 775--884. Springer, New York, 2010.

\bibitem[Dev84]{devlin:constructibility}
K.~J. Devlin.
\newblock {\em Constructibility}.
\newblock Perspectives in Mathematical Logic. Springer--Verlag, New York, 1984.

\bibitem[Dod82]{dodd:coremodel}
A.~J. Dodd.
\newblock {\em The Core Model}, volume~61 of {\em London Mathematical Society
  Lecture Note Series}.
\newblock Cambridge University Press, Cambridge, 1982.

\bibitem[Fen90]{feng:ramsey}
Q.~Feng.
\newblock A hierarchy of {R}amsey cardinals.
\newblock {\em Ann. Pure Appl. Logic}, 49(3):257--277, 1990.

\bibitem[Gai74]{gaifman:ultrapowers}
H.~Gaifman.
\newblock Elementary embeddings of models of set-theory and certain
  subtheories.
\newblock In {\em Axiomatic set theory (Proc. Sympos. Pure Math., Vol. XIII,
  Part II, Univ. California, Los Angeles, Calif., 1967)}, pages 33--101. Amer.
  Math. Soc., Providence R.I., 1974.

\bibitem[GJ10]{gitman:ramseyindes}
V.~Gitman and T.~Johnstone.
\newblock Indestructibility for {R}amsey-like cardinals.
\newblock 2010.
\newblock In preparation.

\bibitem[GW10]{gitman:welch}
V.~Gitman and P.~D. Welch.
\newblock Ramsey-like cardinals {II}.
\newblock {\em The Journal of Symbolic Logic}, 2010.
\newblock To appear.

\bibitem[Ham07]{hamkins:book}
J.~D. Hamkins.
\newblock {\em Forcing and Large Cardinals}.
\newblock 2007.
\newblock Manuscript.

\bibitem[Jec03]{jech:settheory}
T.~Jech.
\newblock {\em Set Theory}.
\newblock Springer Monographs in Mathematics. Springer--Verlag, New York, third
  edition, 2003.

\bibitem[Kan03]{kanamori:higher}
A.~Kanamori.
\newblock {\em The Higher Infinite}.
\newblock Springer Monographs in Mathematics. Springer--Verlag, New York,
  second edition, 2003.

\bibitem[Kle78]{kleinberg:ineffable}
E.~M. Kleinberg.
\newblock A combinatorial characterization of normal {$M$}-ultrafilters.
\newblock {\em Adv. in Math.}, 30(2):77--84, 1978.

\bibitem[Kun70]{kunen:ultrapowers}
K.~Kunen.
\newblock Some applications of iterated ultrapowers in set theory.
\newblock {\em Ann. Math. Logic}, 1:179--227, 1970.

\bibitem[Mit79]{mitchell:ramsey}
W.~J. Mitchell.
\newblock Ramsey cardinals and constructibility.
\newblock {\em Journal of Symbolic Logic}, 44(2):260--266, 1979.

\bibitem[WS10]{welch:ramsey}
P.~D. Welch and I.~Sharpe.
\newblock Greatly {E}rdos cardinals and some generalizations to the {C}hang and
  {R}amsey properties.
\newblock 2010.
\newblock Submitted Annals of Pure and Applied Logic.

\end{thebibliography}
\end{document}